\documentclass[12 pt]{amsart}
\usepackage[english]{babel}
\usepackage[letterpaper,top=2cm,bottom=2cm,left=3cm,right=3cm,marginparwidth=1.75cm]{geometry}
\usepackage{amsthm}
\usepackage{amssymb}
\usepackage{comment}
\usepackage{amsmath}
\usepackage{graphicx}
\usepackage{tikz}
\usepackage{tkz-graph}
\usetikzlibrary{shapes}
\usepackage{float}
\usepackage{verbatim}
\usepackage{subcaption}

\newtheorem{theorem}{Theorem}
\newtheorem{lemma}[theorem]{Lemma}

\newtheorem{corollary}[theorem]{Corollary}

\DeclareMathOperator{\Pa}{Parent}
\DeclareMathOperator{\Desc}{Desc}
\DeclareMathOperator{\children}{Children}
\DeclareMathOperator{\vs}{vs}
\DeclareMathOperator{\pw}{pw}

\title{The Isoperimetric Peak of Complete Trees}
\author{Anthony Bonato, Lazar Mandic, Trent G.\ Marbach, Matthew Ritchie}

\begin{document}
\maketitle

\begin{abstract}
We give exact values and bounds on the isoperimetric peak of complete trees, improving on known results. For the complete $q$-ary tree of depth $d$, if $q\ge 5$, then we find that the isoperimetric peak equals $d$, completing an open problem. In the case that $q$ is 3 or 4, we determine the value up to three values, and in the case $q=2$, up to a logarithmic additive factor. Our proofs use novel compression techniques, including left, down, and aeolian compressions. We apply our results to show that the vertex separation number and the isoperimetric peak of a graph may be arbitrarily far apart as a function of the order of the graph and give new bounds on the pathwidth and pursuit-evasion parameters on complete trees.
\end{abstract}

\section{Introduction}

Isoperimetry studies the relationship between an object's perimeter and its volume and is an ancient mathematical problem with modern applications in differential geometry, analysis, and graph theory.  For graphs, isoperimetry studies the comparison between the number of vertices in some subset of vertices and the bordering vertices of this subset. Isoperimetry appears in many topics within graph theory, such as in spectral graph theory \cite{MR0782626,fan}, independence number \cite{MR4079359}, network connectivity \cite{conn}, and pursuit-evasion games \cite{MR4643801,bonato2023kvisibility}. For more background on isoperimetric inequalities in graph theory, see \cite{fan}.

Define the \emph{vertex-border} of a set of vertices $S \subseteq V(G)$ as 
\(\delta(S) = \{v \in V(G)\backslash S : N(v) \cap S \neq \emptyset\}. \)
The \emph{vertex-isoperimetric parameter on graph $G$ with volume $s$} is defined as $\Phi_V(G,s) = \min_{S:|S|=s} |\delta(S)|$, 
The \emph{vertex-isoperimetric peak} of $G$, written $\Phi(G)$, is the maximum cardinality of $\Phi_V(G,s)$ over all $s$. 
That is, \[\Phi(G) = \max\limits_{0\leq k \leq |V(G)|} \Phi_V(G,s).\]
We will restrict ourselves to vertex-isoperimetry in this paper, although we do note that an edge-version of the isoperimetric peak is studied, where the border of a set is instead the edges that have exactly one vertex in the set. 

Our present focus is on the vertex isoperimetry of complete trees. The \emph{complete $q$-ary} tree of depth $d$ is the rooted tree where every vertex of distance less than $d$ from the root has $q$ children and where the vertices of distance $d$ from the root have degree $1$. 
The lower bound of the isoperimetric peak of complete $q$-ary trees was first studied by Otachi and Yamazaki, who showed the following result. 

\begin{theorem}\cite{MR2410445} For a complete $q$-ary tree $T$ of depth $d$, 
\[
    \Phi_V(T) \geq \frac{d \log{q} - (q + 6 + 2 \log{d})}{q + 6 + 2\log{d}}.
\] 
\end{theorem}

Bharadwaj and Chandran tightened the result, yielding a lower bound that is linear to the upper bound when $q$ is constant. 

\begin{theorem}\cite{MR2502192} For a complete $q$-ary tree $T$ of depth $d$, when $q\geq 3$ there is a constant $c$ such that $c \frac{d}{\sqrt{q}} \leq \Phi_V(T) \leq d,$ and when $q=2$, $c d \leq \Phi_V(T) \leq d.$ 
\end{theorem}

Vrt'{o} found a linear lower bound independent of $q$, thus showing that $\Phi_V(T) = \Theta(d)$ for any family of complete trees for the first time. 

 \begin{theorem} \cite{MR2579861}
For a complete $q$-ary tree $T$ of depth $d$, $\frac{3}{40}(d-2) \leq \Phi_V(T) \leq d.$
 \end{theorem}
 
We improve these results by providing a lower bound on the isoperimetric peak of complete trees that matches the known upper bound when $q \geq 5$, and for $q \in \{3,4\}$, we give both an upper and lower bound that differ from each other by at most $3$, representing a substantial improvement on previous results. We also improve the lower bound in the case that $q=2$, although the lower bound we find differs from the upper bound by $\log_2(d)/2+4$. 
See Section~\ref{sec:upperbounds} for our discussion on upper bounds. 

Our main results are contained in the following three theorems.

\begin{theorem} \label{thm:both_q5}
For $T$ a complete $q$-ary tree of depth $d$ with $q\geq 5$, $\Phi_V(T) =  d.$
\end{theorem}

\begin{theorem} \label{thm:both_q34}
For $T$ a complete $q$-ary tree of depth $d$ with $q \in \{3,4\}$ and for some $\varepsilon \in [1-\log_q(2), 3]$, we have that $
\Phi_V(T) = d - \log_q(d) + \varepsilon.$
\end{theorem} 

\begin{theorem} \label{thm:both_q2}
For $T$ a complete binary tree of depth $d$,
\[ \frac{d - \log_2(d)-3-\log_2(3)}{2} \leq \Phi_V(T) \leq  \lceil d/2 \rceil. \]
\end{theorem} 

We will also explore how the vertex-isoperimetric parameter is related to a number of other well-studied graph parameters, such as the vertex separation number and pathwidth. In particular, we show that contrary to intuition, the vertex separation number and isoperimetric peak can be arbitrarily far apart. See Sections~\ref{sec:upperbounds} and \ref{sec:applications} for discussion.

We finish the introduction with notation. All graphs considered are simple, undirected, and finite. For background on graph theory, the reader is directed to \cite{west}. Let $T$ be a complete $q$-ary tree of depth $s$. The vertices in $T$ of distance $d$ from the root have degree $1$ and are referred to as \emph{leaves} of $T$. For two adjacent vertices $v$ and $w$ with $v$ closer to the root than $w$,  $v$ is the \emph{parent} of $w$, and $w$ is a \emph{child} of $v$. The parent of $w$ is denoted $\Pa(w)$, and the set of all children of $u$ is denoted $\children(u)$. A vertex $w$ is a \emph{descendant} of a vertex $v$ when the unique path from $w$ to the root contains $v$. The set of descendants of $v$ is written as $\Desc(v)$.  Note in particular that $v$ is its own descendant, and so $v \in \Desc(v)$.
For $q \geq 2$, the cardinality of a complete $q$-ary tree of depth $d$ can be calculated by the geometric sum $t_d = \frac{q^{d+1}-1}{q-1}$.
Two vertices $u,v$ are on the same \emph{level} $L_i$ of $T$ if they have the same distance, $i$, to the root vertex. 

We can order the vertices $V(T)$ using a particular breadth-first transversal, and we say that vertex $u$ is to the \emph{left} of $v$ if $u$ and $v$ are on the same level, and $u$ comes before $v$ in this ordering. The terminology of \emph{right}, \emph{left-most}, and \emph{right-most} follow in the obvious fashion. Under this ordering, the root vertex of $V(T)$ is the smallest in the ordering, a vertex in $L_i$ is smaller in the ordering than a vertex in $L_j$ when $i < j$, and if $u$ and $v$ are in the same level with $u$ before $v$ in the ordering, then all the children of $u$ are smaller in the ordering than the children of $v$. 

\section{Compressions on $q$-ary Trees}

Compressions have been a useful tool in the study of isoperimetry on graphs \cite{MR1455181, MR1444247,MR1082842,MR2035509, TriGrids}. 
We begin our discussion on compressions with an illustrative example, which is presented visually in Figure~\ref{fig:12compression}. 
For the square grid $[k]^2$, define a \emph{$1$-compression} on a $S\subseteq V$ as the subset $S'$ derived from $S$ by recursively finding any $x$ and $y$ such that $(x,y)\in S$ and $(x-1,y) \notin S$, and then removing $(x,y)$ from $S$ and replacing it with $(x-1,y)$. 
A \emph{$2$-compression} works similarly, but where $(x,y-1)$ is used in place of $(x-1,y)$.  

\begin{figure}[h]
    \centering
    \begin{tikzpicture}
    \node[fill=black, circle, inner sep=2pt] (v0) at (0,0) {};
    \node[draw=black, fill=white, circle, inner sep=2pt] (v1) at (1, 0) {};
    \node[fill=black, circle, inner sep=2pt] (v2) at (2, 0) {};
    \node[draw=black, fill=white, circle, inner sep=2pt] (v3) at (0, 1) {};
    \node[draw=black, fill=white, circle, inner sep=2pt] (v4) at (1, 1) {};
    \node[fill=black, circle, inner sep=2pt] (v5) at (2, 1) {};
    \node[fill=black, circle, inner sep=2pt] (v6) at (0, 2) {};
    \node[draw=black, fill=white, circle, inner sep=2pt] (v7) at (1, 2) {};
    \node[draw=black, fill=white, circle, inner sep=2pt] (v8) at (2, 2) {};
    \draw (v0) -- (v1);
    \draw (v0) -- (v3);
    \draw (v1) -- (v2);
    \draw (v1) -- (v4);
    \draw (v2) -- (v5);
    \draw (v3) -- (v4);
    \draw (v3) -- (v6);
    \draw (v4) -- (v5);
    \draw (v4) -- (v7);
    \draw (v5) -- (v8);
    \draw (v6) -- (v7);
    \draw (v7) -- (v8);

    \draw[->] (3,1) -- (4,1);

    \node[fill=black, circle, inner sep=2pt] (w0) at (5,0) {};
    \node[draw=black, fill=white, circle, inner sep=2pt] (w1) at (6, 0) {};
    \node[fill=black, circle, inner sep=2pt] (w2) at (7, 0) {};
    \node[fill=black, circle, inner sep=2pt] (w3) at (5, 1) {};
    \node[draw=black, fill=white, circle, inner sep=2pt] (w4) at (6, 1) {};
    \node[fill=black, circle, inner sep=2pt] (w5) at (7, 1) {};
    \node[draw=black, fill=white, circle, inner sep=2pt] (w6) at (5, 2) {};
    \node[draw=black, fill=white, circle, inner sep=2pt] (w7) at (6, 2) {};
    \node[draw=black, fill=white, circle, inner sep=2pt] (w8) at (7, 2) {};
    \draw (w0) -- (w1);
    \draw (w0) -- (w3);
    \draw (w1) -- (w2);
    \draw (w1) -- (w4);
    \draw (w2) -- (w5);
    \draw (w3) -- (w4);
    \draw (w3) -- (w6);
    \draw (w4) -- (w5);
    \draw (w4) -- (w7);
    \draw (w5) -- (w8);
    \draw (w6) -- (w7);
    \draw (w7) -- (w8);
    
    \draw[->] (8,1) -- (9,1);

    \node[fill=black, circle, inner sep=2pt] (w0) at (10,0) {};
    \node[fill=black, circle, inner sep=2pt] (w1) at (11, 0) {};
    \node[draw=black, fill=white, circle, inner sep=2pt] (w2) at (12, 0) {};
    \node[fill=black, circle, inner sep=2pt] (w3) at (10, 1) {};
    \node[fill=black, circle, inner sep=2pt] (w4) at (11, 1) {};
    \node[draw=black, fill=white, circle, inner sep=2pt] (w5) at (12, 1) {};
    \node[draw=black, fill=white, circle, inner sep=2pt] (w6) at (10, 2) {};
    \node[draw=black, fill=white, circle, inner sep=2pt] (w7) at (11, 2) {};
    \node[draw=black, fill=white, circle, inner sep=2pt] (w8) at (12, 2) {};
    \draw (w0) -- (w1);
    \draw (w0) -- (w3);
    \draw (w1) -- (w2);
    \draw (w1) -- (w4);
    \draw (w2) -- (w5);
    \draw (w3) -- (w4);
    \draw (w3) -- (w6);
    \draw (w4) -- (w5);
    \draw (w4) -- (w7);
    \draw (w5) -- (w8);
    \draw (w6) -- (w7);
    \draw (w7) -- (w8);
    
\end{tikzpicture}
    \caption{An example of a $1$-compression and then a $2$-compression on a subset of vertices in the square grid, where the subset of vertices are the four vertices represented by filled vertices.}
\label{fig:12compression}
\end{figure}
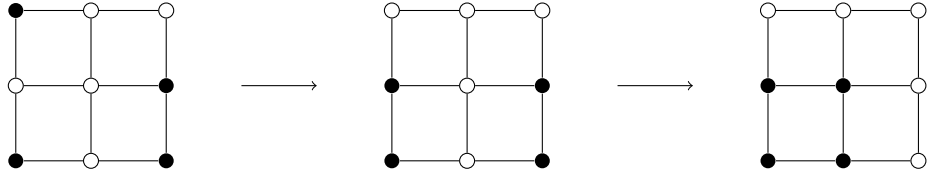

These compressions were defined in greater generality by Bollob\'{a}s and Leader \cite{MR1082842}, who showed that in a number of graph classes, applying a number of compressions to a subset $S$ of the vertices could be compressed to a set $S'$ with $|S'|=|S|$ and where $|\delta(S')| \leq |\delta(S)|$. 
By recursively applying different compressions, it is possible to find a set $S'$ with $|S'| = |S|$ and where $|\delta(S')|=\Phi_V(G,|S|)$. 

The kind of compressions used previously cannot be applied to trees; however, we will use a natural generalization of this terminology. 
We say $S'\subseteq V(T)$ is a \emph{compression} of $S \subseteq V(T)$ if $|S'|=|S|$ and $|\delta(S')| \leq |\delta(S)|$. 
For a set $X$, let $\mathcal{P}(X)$ be the power set of $X$. 
A function $f:\mathcal{P}(V) \rightarrow \mathcal{P}(V)$ is called a \emph{compression function} if, for all $S\in \mathcal{P}(V)$, $f(S)$ is a compression of $S$, and where recursively applying $f$ to $S$ can only yield $S$ if $f(S)=S$. 
Given a compression function $f$ and a subset of vertices $S\subseteq V$, we say $S$ is \emph{$f$-compressed} if $f(S) = S$.

\subsection{Left-compressions}

In this subsection, we will define mappings that `swap' certain subtrees within a subset of vertices $S$, which we will show are compressions. 
Recursively applying these compressions will yield a new subset of vertices with the same number of vertices per level of the tree as in $S$, but where all the vertices are pushed as far `left' as possible. 
See Figure~\ref{fig:left_compression_example} for an example. 

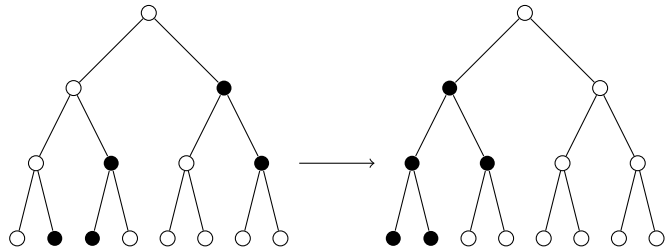
\begin{figure}[h]
    \centering
    \begin{tikzpicture}
    \node[draw=black, fill=white, circle, inner sep=2pt] (u) at (0, 0) {};
    \node[draw=black, fill=white, circle, inner sep=2pt] (u0) at (-1, -1) {};
    \node[fill=black, circle, inner sep=2pt] (u1) at (1, -1) {};
    \node[draw=black, fill=white, circle, inner sep=2pt] (u00) at (-1.5, -2) {};
    \node[fill=black, circle, inner sep=2pt] (u01) at (-0.5, -2) {};
    \node[draw=black, fill=white, circle, inner sep=2pt] (u10) at (0.5, -2) {};
    \node[fill=black, circle, inner sep=2pt] (u11) at (1.5, -2) {};
    \node[draw=black, fill=white, circle, inner sep=2pt] (u000) at (-1.75, -3) {};
    \node[fill=black, circle, inner sep=2pt] (u001) at (-1.25, -3) {};
    \node[fill=black, circle, inner sep=2pt] (u010) at (-0.75, -3) {};
    \node[draw=black, fill=white, circle, inner sep=2pt] (u011) at (-0.25, -3) {};
    \node[draw=black, fill=white, circle, inner sep=2pt] (u100) at (0.25, -3) {};
    \node[draw=black, fill=white, circle, inner sep=2pt] (u101) at (0.75, -3) {};
    \node[draw=black, fill=white, circle, inner sep=2pt] (u110) at (1.25, -3) {};
    \node[draw=black, fill=white, circle, inner sep=2pt] (u111) at (1.75, -3) {};
    \draw (u) -- (u0);
    \draw (u) -- (u1);
    \draw (u0) -- (u00);
    \draw (u0) -- (u01);
    \draw (u1) -- (u10);
    \draw (u1) -- (u11);
    \draw (u00) -- (u000);
    \draw (u00) -- (u001);
    \draw (u01) -- (u010);
    \draw (u01) -- (u011);
    \draw (u10) -- (u100);
    \draw (u10) -- (u101);
    \draw (u11) -- (u110);
    \draw (u11) -- (u111);

    \draw[->] (2,-2) -- (3,-2);
    
    \node[draw=black, fill=white, circle, inner sep=2pt] (u) at (5, 0) {};
    \node[fill=black, circle, inner sep=2pt] (u0) at (4, -1) {};
    \node[draw=black, fill=white, circle, inner sep=2pt] (u1) at (6, -1) {};
    \node[fill=black, circle, inner sep=2pt] (u00) at (3.5, -2) {};
    \node[fill=black, circle, inner sep=2pt] (u01) at (4.5, -2) {};
    \node[draw=black, fill=white, circle, inner sep=2pt] (u10) at (5.5, -2) {};
    \node[draw=black, fill=white, circle, inner sep=2pt] (u11) at (6.5, -2) {};
    \node[fill=black, circle, inner sep=2pt] (u000) at (3.25, -3) {};
    \node[fill=black, circle, inner sep=2pt] (u001) at (3.75, -3) {};
    \node[draw=black, fill=white, circle, inner sep=2pt] (u010) at (4.25, -3) {};
    \node[draw=black, fill=white, circle, inner sep=2pt] (u011) at (4.75, -3) {};
    \node[draw=black, fill=white, circle, inner sep=2pt] (u100) at (5.25, -3) {};
    \node[draw=black, fill=white, circle, inner sep=2pt] (u101) at (5.75, -3) {};
    \node[draw=black, fill=white, circle, inner sep=2pt] (u110) at (6.25, -3) {};
    \node[draw=black, fill=white, circle, inner sep=2pt] (u111) at (6.75, -3) {};
    \draw (u) -- (u0);
    \draw (u) -- (u1);
    \draw (u0) -- (u00);
    \draw (u0) -- (u01);
    \draw (u1) -- (u10);
    \draw (u1) -- (u11);
    \draw (u00) -- (u000);
    \draw (u00) -- (u001);
    \draw (u01) -- (u010);
    \draw (u01) -- (u011);
    \draw (u10) -- (u100);
    \draw (u10) -- (u101);
    \draw (u11) -- (u110);
    \draw (u11) -- (u111);
\end{tikzpicture}
    \caption{An example of the action of the left-compression on a subset of vertices in a complete binary tree.}
\label{fig:left_compression_example}
\end{figure}

 Two vertices $u,v \in V(T)$ are \emph{swappable} in $S$ if $u$ to the left of $v$, and either:
 \begin{enumerate}
     \item $u\notin S$ and $v \in S$;  or 
     \item $u,v \notin S$, $\children(u)\cap S = \emptyset$, and $\children(v)\cap S \neq \emptyset$.
 \end{enumerate} 
 We say level $i$ of $S$ is \emph{left-ordered} if $S$ does not contain a swappable pair in level $i$. 
 Further, we say that $S$ itself is \emph{left-ordered} if each of its levels is left-ordered. 
 We note here that condition (2) could be dropped, and the left-ordered definitions would remain unchanged; however, we require condition (2) to ensure that the mappings we perform will be compressions. 

Write the vertices in $\Desc(u)$ and $\Desc(v)$ as $(u^{(1)}, u^{(2)}, \ldots)$ and $(v^{(1)}, v^{(2)}, \ldots)$, respectively, where we assume than $u^{(i)}$ comes before $u^{(j)}$ in the breadth-first ordering for $i<j$, and likewise for $v^{(i)}$ and $v^{(j)}$. 
Define the \emph{treeswap of $S$ between $u$ and $v$} as the subset of vertices $S'$, where:
\begin{enumerate}
    \item $x \in S'$ if $x \in V(T)\setminus (\Desc(u) \cup \Desc(v))$; 
    \item $v^{(i)} \in S'$ if $u^{(i)} \in S$; and
    \item $u^{(i)} \in S'$ if $v^{(i)} \in S$. 
\end{enumerate}
See Figure~\ref{fig:treeswap} for an illustration of a treeswap. 

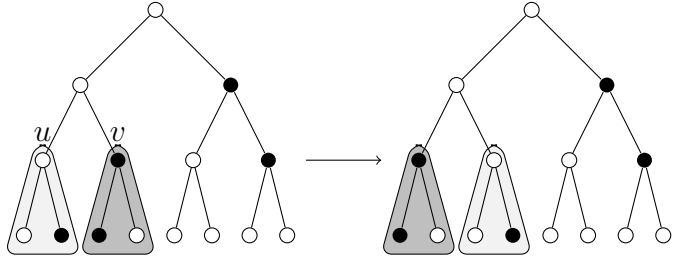
\begin{figure}[h]
    \centering
    \begin{tikzpicture}

    \draw [rounded corners,fill=gray!10] (-2,-3.25)--(-1,-3.25)--(-1.4,-1.8)--(-1.6,-1.8)--cycle;
    \draw [rounded corners,fill=gray!50] (-1,-3.25)--(0,-3.25)--(-0.4,-1.8)--(-0.6,-1.8)--cycle;
    \draw [rounded corners,fill=gray!50] (3,-3.25)--(4,-3.25)--(3.6,-1.8)--(3.4,-1.8)--cycle;
    \draw [rounded corners,fill=gray!10] (4,-3.25)--(5,-3.25)--(4.6,-1.8)--(4.4,-1.8)--cycle;
    
    \node[draw=black, fill=white, circle, inner sep=2pt] (u) at (0, 0) {};
    \node[draw=black, fill=white, circle, inner sep=2pt] (u0) at (-1, -1) {};
    \node[fill=black, circle, inner sep=2pt] (u1) at (1, -1) {};
    \node[draw=black, fill=white, circle, inner sep=2pt, label=above:$u$] (u00) at (-1.5, -2) {};
    \node[fill=black, circle, inner sep=2pt, label=above:$v$] (u01) at (-0.5, -2) {};
    \node[draw=black, fill=white, circle, inner sep=2pt] (u10) at (0.5, -2) {};
    \node[fill=black, circle, inner sep=2pt] (u11) at (1.5, -2) {};
    \node[draw=black, fill=white, circle, inner sep=2pt] (u000) at (-1.75, -3) {};
    \node[fill=black, circle, inner sep=2pt] (u001) at (-1.25, -3) {};
    \node[fill=black, circle, inner sep=2pt] (u010) at (-0.75, -3) {};
    \node[draw=black, fill=white, circle, inner sep=2pt] (u011) at (-0.25, -3) {};
    \node[draw=black, fill=white, circle, inner sep=2pt] (u100) at (0.25, -3) {};
    \node[draw=black, fill=white, circle, inner sep=2pt] (u101) at (0.75, -3) {};
    \node[draw=black, fill=white, circle, inner sep=2pt] (u110) at (1.25, -3) {};
    \node[draw=black, fill=white, circle, inner sep=2pt] (u111) at (1.75, -3) {};
    \draw (u) -- (u0);
    \draw (u) -- (u1);
    \draw (u0) -- (u00);
    \draw (u0) -- (u01);
    \draw (u1) -- (u10);
    \draw (u1) -- (u11);
    \draw (u00) -- (u000);
    \draw (u00) -- (u001);
    \draw (u01) -- (u010);
    \draw (u01) -- (u011);
    \draw (u10) -- (u100);
    \draw (u10) -- (u101);
    \draw (u11) -- (u110);
    \draw (u11) -- (u111);

    \draw[->] (2,-2) -- (3,-2);

    \node[draw=black, fill=white, circle, inner sep=2pt] (u) at (5, 0) {};
    \node[draw=black, fill=white, circle, inner sep=2pt] (u0) at (4, -1) {};
    \node[fill=black, circle, inner sep=2pt] (u1) at (6, -1) {};
    \node[fill=black, circle, inner sep=2pt] (u00) at (3.5, -2) {};
    \node[draw=black, fill=white, circle, inner sep=2pt] (u01) at (4.5, -2) {};
    \node[draw=black, fill=white, circle, inner sep=2pt] (u10) at (5.5, -2) {};
    \node[fill=black, circle, inner sep=2pt] (u11) at (6.5, -2) {};
    \node[fill=black, circle, inner sep=2pt] (u000) at (3.25, -3) {};
    \node[draw=black, fill=white, circle, inner sep=2pt] (u001) at (3.75, -3) {};
    \node[draw=black, fill=white, circle, inner sep=2pt] (u010) at (4.25, -3) {};
    \node[fill=black, circle, inner sep=2pt] (u011) at (4.75, -3) {};
    \node[draw=black, fill=white, circle, inner sep=2pt] (u100) at (5.25, -3) {};
    \node[draw=black, fill=white, circle, inner sep=2pt] (u101) at (5.75, -3) {};
    \node[draw=black, fill=white, circle, inner sep=2pt] (u110) at (6.25, -3) {};
    \node[draw=black, fill=white, circle, inner sep=2pt] (u111) at (6.75, -3) {};
    \draw (u) -- (u0);
    \draw (u) -- (u1);
    \draw (u0) -- (u00);
    \draw (u0) -- (u01);
    \draw (u1) -- (u10);
    \draw (u1) -- (u11);
    \draw (u00) -- (u000);
    \draw (u00) -- (u001);
    \draw (u01) -- (u010);
    \draw (u01) -- (u011);
    \draw (u10) -- (u100);
    \draw (u10) -- (u101);
    \draw (u11) -- (u110);
    \draw (u11) -- (u111);
\end{tikzpicture}
    \caption{An example of a treeswap of $S$ between vertices $u$ and $v$, where $S$ is the set of filled vertices.}
    \label{fig:treeswap}
\end{figure}

\begin{lemma} \label{lem:swappable_for_compression}
Let $u,v\in L_i$ be swappable, and suppose each level before level $i$ is left-ordered. 
We then have that the treeswap of $S$ between $u$ and $v$, $S'$ is a compression of $S$. 
\end{lemma}
\begin{proof}
We begin by noting that by the definition of the treeswap, 
$S \cap (V(T) \setminus (\Desc(u) \cup \Desc(v)) = S' \cap (V(T) \setminus (\Desc(u) \cup \Desc(v))$. 
Since $\Pa(u)$ and $\Pa(v)$ are the only vertices in $V(T) \setminus (\Desc(u) \cup \Desc(v))$ with a neighbour not in $V(T) \setminus (\Desc(u) \cup \Desc(v))$, we have that $\delta(S) \cap (V(T) \setminus (\Desc(u) \cup \Desc(v) \cup \{\Pa(u),\Pa(v)\})$ equals $\delta(S') \cap (V(T) \setminus (\Desc(u) \cup \Desc(v) \cup \{\Pa(u),\Pa(v)\}).$

Notice that $u^{(i)} \in \Desc(u)$ is adjacent to  $u^{(j)} \in \Desc(u)$ if and only if $v^{(i)}\in \Desc(v)$ is adjacent to  $v^{(j)} \in \Desc(v)$. As such, if we supposed $u^{(i)}\in \Desc{u}\setminus \{u\}$ has $ u^{(i)}\in \delta(S)$, then $u^{(i)}\notin S$ and there is some $u^{(j)}$ adjacent to $u^{(i)}$ with $u^{(j)}\in S$, from which it follows that $v^{(i)}\notin S'$ and $v^{(j)}\in S'$, so that $v^{(i)}\in \delta(S')$. 
As such, $|(\Desc(u) \setminus \{u\}) \cap \delta(S)| = |(\Desc(v) \setminus \{v\}) \cap \delta(S')|$. 
This also holds if we swap the roles of $u$ and $v$, and so we also have that $|(\Desc(v) \setminus \{v\}) \cap \delta(S)| = |(\Desc(u) \setminus \{u\}) \cap \delta(S')|$. 

Hence, there is a bijection between $\delta(S) \setminus \{u,v,\Pa(u),\Pa(v)\}$ and $\delta(S') \setminus \{u,v,\Pa(u),\Pa(v)\}$.
If $|\delta(S')|> |\delta(S)|$, then it must be that $$|\delta(S') \cap \{u,v,\Pa(u),\Pa(v)\}|>|\delta(S) \cap \{u,v,\Pa(u),\Pa(v)\}|.$$
The remaining arguments of this proof will show that this cannot be the case. 
Since $u$ and $v$ are swappable, either $u\notin S$ and $v \in S$, or $u,v \notin S$, $\children(u)\cap S = \emptyset$ and $\children(v)\cap S \neq \emptyset$. 
We begin with the first case that $u\notin S$ and $v \in S$. 
This breaks into four cases. 

\medskip
\noindent \emph{Case 1}: If $\Pa(u),\Pa(v)\in S$, then   
$\delta(S) \cap \{u,v,\Pa(u),\Pa(v)\} = \{u\}$ and 
$\delta(S') \cap \{u,v,\Pa(u),\Pa(v)\} = \{v\}$.
\medskip

\noindent \emph{Case 2}: If $\Pa(u)\in S$ and $\Pa(v)\notin S$, then   
$\delta(S) \cap \{u,v,\Pa(u),\Pa(v)\} = \{u, \Pa(v)\}$ and 
$\delta(S') \cap \{u,v,\Pa(u),\Pa(v)\} \subseteq \{v, \Pa(v)\}$. 

\medskip
\noindent \emph{Case 3}: It cannot be that $\Pa(u)\notin S$ and $\Pa(v)\in S$.
Otherwise, level $i-1$ is not left-ordered as we required. 
\medskip

\noindent \emph{Case 4}: For the final case, suppose we have that $\Pa(u),\Pa(v)\notin S$.
As $\Pa(u),\Pa(v)$ were not a swappable pair since level $i-1$ was left-ordered, it must be that either $\children(\Pa(u)) \cap S \neq \emptyset$ or $\children(\Pa(v)) \cap S = \emptyset$. 
The latter is not true as $v \in \children(\Pa(v))\cap S$, so it must be that $\children(\Pa(u)) \cap S \neq \emptyset$. 
Since $\Pa(u) \notin S$, this gives that $\Pa(u) \in \delta(S)$. 
We also have that $\Pa(v) \in \delta(S)$ since $v\in S$ and $\Pa(v)\notin S$.  
Since $u,\Pa(u), \Pa(v) \notin S$, it follows that $v\in \delta(S')$ if and only if $u \in \delta(S)$. 
Hence, either $\delta(S) \cap \{u,v,\Pa(u),\Pa(v)\} = \{u,\Pa(u),\Pa(v)\}$ and also $$\delta(S') \cap \{u,v,\Pa(u),\Pa(v)\} \subseteq \{v,\Pa(u),\Pa(v)\},$$ or instead we have that 
$\delta(S) \cap \{u,v,\Pa(u),\Pa(v)\} = \{\Pa(u),\Pa(v)\}$ and $\delta(S') \cap \{u,v,\Pa(u),\Pa(v)\} \subseteq \{\Pa(u),\Pa(v)\}$. 
For both of these, $|\delta(S') \cap \{u,v,\Pa(u),\Pa(v)\}| \leq |\delta(S) \cap \{u,v,\Pa(u),\Pa(v)\}|$, and so the last subcase is complete. 

We now assume the second case that $u,v \notin S$, $\children(u)\cap S = \emptyset,$ and $\children(v)\cap S \neq \emptyset$. 
If $\Pa(v)\in S$, then $\Pa(u)\in S$, or else $\Pa(u),\Pa(u)$ was a swappable pair in a left-ordered level, forming a contradiction. 
In this case, $\delta(S) \cap \{u,v,\Pa(u),\Pa(v)\} = \{u,v\} =  \delta(S') \cap \{u,v,\Pa(u),\Pa(v)\}$, and so the required condition holds.  
Therefore, we may assume that $\Pa(v)\notin S$. 
We then have that $v \in \delta(S)$, $u \in \delta(S')$, and $v \notin \delta(S')$. 
As such, $\{v\} \subseteq \delta(S) \cap \{u,v\}$ and $\delta(S') \cap \{u,v\} \subseteq \{u\}$. 
Also, $\delta(S') \cap \{\Pa(u),\Pa(v)\} = \delta(S) \cap \{\Pa(u),\Pa(v)\}$. 
It follows that $|\delta(S') \cap \{u,v,\Pa(u),\Pa(v)\}| \leq |\delta(S) \cap \{u,v,\Pa(u),\Pa(v)\}|$, and so the last case is complete. 
\end{proof}

We define the function $\mathrm{left}:\mathcal{P}(V(T))\rightarrow \mathcal{P}(V(T))$ to be the function that takes a subset of vertices $S$, and performs the treeswap on the left-most swappable pair $u,v\in L_i$ in $S$, where $i$ is such that each level before level $i$ is left-ordered. 
As a consequence of Lemma~\ref{lem:swappable_for_compression}, to show that the function $\mathrm{left}$ is a compression function, we are required to show that the recursive application of the function to some subset of vertices $S$ can only yield $S$ again if $\mathrm{left}(S)=S$. 
Define the strict partial ordering $<$ on subsets of $V(T)$ as $A < B$ when: 
\begin{enumerate}
    \item $A$ and $B$ have the same number of vertices in every level of $T$; and 
    \item there is a vertex $v \in V(T)$ such that $v\in A$, $v\notin B$, and for each vertex $w$ before $v$ in the ordering of the vertices of $T$, either $w$ is in both $A$ and $B$, or is in neither.
\end{enumerate}

\begin{lemma} \label{lem:not_left_compressed_means_swappable}
For a $S\subseteq V(T)$, if there exist a $U\subseteq V(T)$ with $U<S$, then $\mathrm{left}(S)<S$. 
\end{lemma} 
\begin{proof}
Suppose $S$ and $U$ are subsets of $V(T)$ with $U<S$.  
By the definition of the partial order, there must be some $u'\in U$ with $u' \notin S$ so that for each vertex $w$ before $u'$ in the ordering on the tree, $w$ is in both $S$ and $U$, or $w$ is in neither. 
Also, each level of $S$ and $U$ must contain the same number of vertices. 
As such, there must be some vertex $v'$ in the same level as $u'$ with $v' \in S$ and $v' \notin U$, with $u'$ to the left of $v'$. 
However, now there is at least one pair $u',v'$ that is swappable.

Since there is at least one swappable pair, there must be a pair of vertices $u,v$ that are swappable such that no pair of vertices before $u$ and $v$ are swappable. 
Note that $u$ and $v$ may or may not satisfy $u \notin S$ and $v\in S$ by this definition, and we may instead have $u,v\notin S$, $\children(u)\cap S =\emptyset$, and  $\children(v)\cap S \neq \emptyset$. 
Also, $u$ and $v$ may not be on the same level as $u'$ and $v'$; however, we are then guaranteed that each level before the level containing $u$ and $v$ is left-ordered. 
By Lemma~\ref{lem:swappable_for_compression}, the treeswap of $S$ between $u$ and $v$, $\mathrm{left}(S)$, is a compression. 

In the case that $u\notin S$ and $v \in S$, the subsets of vertices $S$ and $S'$ agree on all vertices before $u$ and have $u \in \mathrm{left}(S)$ and $u \notin S$. Since each level in $S$ and $\mathrm{left}(S)$  have the same number of vertices, it follows that $\mathrm{left}(S) < S$. 
In the alternate case that $u,v \notin S$, $\children(u) \cap S = \emptyset$ and $\children(v)\cap S \neq \emptyset$, let $v^{(i)}$ be the left-most child of $v$ that is in $S$. Recall that since $v^{(i)} \in S$, we have $u^{(i)} \in \mathrm{left}(S)$.   
The subsets of vertices $S$ and $\mathrm{left}(S)$ agree on all vertices before $u^{(i)}$ and have $u^{(i)} \in \mathrm{left}(S)$ and $u^{(i)} \notin S$. Since each row in $S$ and $\mathrm{left}(S)$ has the same number of vertices, it follows that $\mathrm{left}(S) < S$. 
This completes the proof.
\end{proof}

As a consequence of Lemma~\ref{lem:not_left_compressed_means_swappable}, if $S$ is $\mathrm{left}$-compressed, then there is no $U$ with $U < S$.
It follows from the definition of the strict partial order that the only $S$ such that there is no $U$ with $U<S$ are those $S$ that are left-ordered. 
\begin{corollary}\label{cor:left_ordered}
    If $S\subseteq V(T)$ is left-compressed, then $S$ is left-ordered. 
\end{corollary}
If we consider the vertex subsets $S$ with $|S|=s$, then we know that $\mathrm{left}(S)$ has the same cardinality as $S$, and $\delta(\mathrm{left}(S))$ is not larger than $\delta(S)$. To analyze the smallest vertex boundary of sets of cardinality $s$, we need to consider those that are left-compressed. 
Corollary~\ref{cor:left_ordered} is useful to us as it allows us to know the structure of left-compressed sets, which will help us prove the main results. 

See Figure~\ref{fig:leftcompressed} for an illustration of a left-ordered subset of vertices. 

\begin{figure}[h]
    \centering
    \begin{tikzpicture}
        
    \node[draw=black, fill=white, circle, inner sep=2pt] (u) at (0, 0) {};
    \node[fill=black, circle, inner sep=2pt] (u0) at (-2, -1) {};
    \node[draw=black, fill=white, circle, inner sep=2pt] (u1) at (2, -1) {};
    \node[fill=black, circle, inner sep=2pt] (u00) at (-3, -2) {};
    \node[fill=black, circle, inner sep=2pt] (u01) at (-1, -2) {};
    \node[draw=black, fill=white, circle, inner sep=2pt] (u10) at (1, -2) {};
    \node[draw=black, fill=white, circle, inner sep=2pt] (u11) at (3, -2) {};
    \node[fill=black, circle, inner sep=2pt] (u000) at (-3.5, -3) {};
    \node[fill=black, circle, inner sep=2pt] (u001) at (-2.5, -3) {};
    \node[draw=black, fill=white, circle, inner sep=2pt] (u010) at (-1.5, -3) {};
    \node[draw=black, fill=white, circle, inner sep=2pt] (u011) at (-0.5, -3) {};
    \node[draw=black, fill=white, circle, inner sep=2pt] (u100) at (0.5, -3) {};
    \node[draw=black, fill=white, circle, inner sep=2pt] (u101) at (1.5, -3) {};
    \node[draw=black, fill=white, circle, inner sep=2pt] (u110) at (2.5, -3) {};
    \node[draw=black, fill=white, circle, inner sep=2pt] (u111) at (3.5, -3) {};
    \node[fill=black, circle, inner sep=2pt] (u0000) at (-3.75, -4) {};
    \node[fill=black, circle, inner sep=2pt] (u0001) at (-3.25, -4) {};
    \node[fill=black, circle, inner sep=2pt] (u0010) at (-2.75, -4) {};
    \node[fill=black, circle, inner sep=2pt] (u0011) at (-2.25, -4) {};
    \node[fill=black, circle, inner sep=2pt] (u0100) at (-1.75, -4) {};
    \node[fill=black, circle, inner sep=2pt] (u0101) at (-1.25, -4) {};
    \node[fill=black, circle, inner sep=2pt] (u0110) at (-0.75, -4) {};
    \node[fill=black, circle, inner sep=2pt] (u0111) at (-0.25, -4) {};
    \node[fill=black, circle, inner sep=2pt] (u1000) at (0.25, -4) {};
    \node[draw=black, fill=white, circle, inner sep=2pt] (u1001) at (0.75, -4) {};
    \node[draw=black, fill=white, circle, inner sep=2pt] (u1010) at (1.25, -4) {};
    \node[draw=black, fill=white, circle, inner sep=2pt] (u1011) at (1.75, -4) {};
    \node[draw=black, fill=white, circle, inner sep=2pt] (u1100) at (2.25, -4) {};
    \node[draw=black, fill=white, circle, inner sep=2pt] (u1101) at (2.75, -4) {};
    \node[draw=black, fill=white, circle, inner sep=2pt] (u1110) at (3.25, -4) {};
    \node[draw=black, fill=white, circle, inner sep=2pt] (u1111) at (3.75, -4) {};
    \draw (u) -- (u0);
    \draw (u) -- (u1);
    \draw (u0) -- (u00);
    \draw (u0) -- (u01);
    \draw (u1) -- (u10);
    \draw (u1) -- (u11);
    \draw (u00) -- (u000);
    \draw (u00) -- (u001);
    \draw (u01) -- (u010);
    \draw (u01) -- (u011);
    \draw (u10) -- (u100);
    \draw (u10) -- (u101);
    \draw (u11) -- (u110);
    \draw (u11) -- (u111);
    \draw (u000) -- (u0000);
    \draw (u000) -- (u0001);
    \draw (u001) -- (u0010);
    \draw (u001) -- (u0011);
    \draw (u010) -- (u0100);
    \draw (u010) -- (u0101);
    \draw (u011) -- (u0110);
    \draw (u011) -- (u0111);
    \draw (u100) -- (u1000);
    \draw (u100) -- (u1001);
    \draw (u101) -- (u1010);
    \draw (u101) -- (u1011);
    \draw (u110) -- (u1100);
    \draw (u110) -- (u1101);
    \draw (u111) -- (u1110);
    \draw (u111) -- (u1111);
    \end{tikzpicture}
    \caption{The left-compressed subset of vertices in a complete binary tree that has one vertex in level 1, two vertices in levels 2 and 3, and nine vertices in level 4.}
    \label{fig:leftcompressed}
\end{figure}
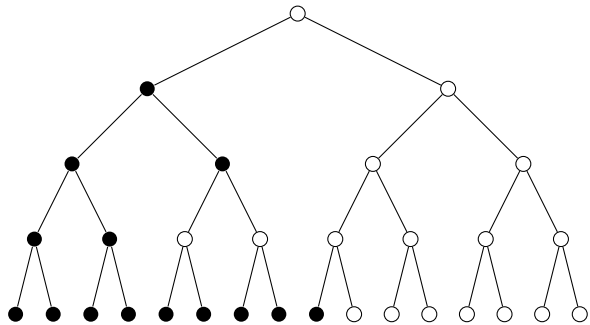

\subsection{Push-down compressions}

In this section, we will observe mappings in which the set of chosen vertices $S$ is pushed down, which means we replace vertices in $S$ with vertices further from the root of the tree. 
We start this endeavour by defining a \emph{$u$--$v$ exchange} of $S$: replace the single vertex $u\in S$ with a vertex $v \notin S$ to form a new set of vertices $S' = (S\setminus \{u\}) \cup \{v\}$.
If $|\delta(S')|\leq |\delta(S)|$, then this operation is a compression. 

While this seems to have no inherent downward push, we notice that any leaf not in $S$ makes a good candidate for the role of $v$, as does any vertex not in $S$ but with all its descendants in $S$. 
Two vertices $u,v\in V(T)$ with $u\in S$ and $v \notin S$ are \emph{down-swappable} in $S$ if either 
\begin{enumerate}
    \item $|\delta(S \cup v) \setminus \{u\})| < |\delta(S)|$, or 
    \item $|\delta(S \cup v) \setminus \{u\})| = |\delta(S)|$ and $v$ is further from the root than $u$. 
\end{enumerate}
We define the function $\mathrm{down}:\mathcal{P}(V(T))\rightarrow \mathcal{P}(V(T))$ to be the function $\mathrm{down}(S) = (S \cup v) \setminus \{u\})$ such that $u$ and $V$ are the first vertices in the ordering on $T$ such that $u$ and $v$ are down-swappable.  

For a set $S\subseteq V(T)$, we may recursively apply the compression function $\mathrm{down}$ to $S$ if there are $u$ and $v$ that are down-swappable. Since each such exchange either reduces the vertex border or results in a strictly greater number of vertices in lower levels, this recursive process must terminate with the down-compressed set $S' \subseteq V(T)$ with $|S'| = |S|$ and $|\delta(S')| \leq |\delta(S)|$.

Note that if $S$ is down-compressed, then 
\begin{enumerate}    
\item For all $u\in S$ and $v \in V(T)\setminus S$, 
\[
| \delta( (S\cup\{v\})\setminus\{u\} )| \geq |\delta(S)|; \text{ and}
\] 
\item for all $u\in S$ and $v \in V(T)\setminus S$ with $v$ further from the root vertex than $u$, 
\[
|\delta((S\cup\{v\})\setminus\{u\})| > |\delta(S)|.
\] 
\end{enumerate}

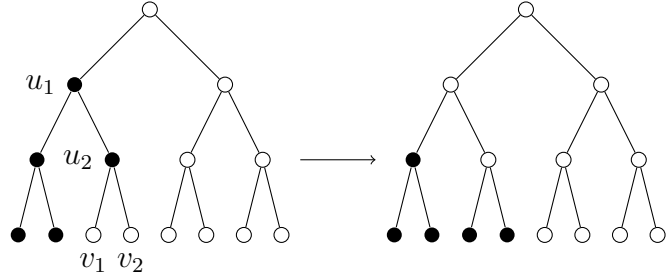
\begin{figure}
    \centering
    \begin{tikzpicture}
    \node[draw=black, fill=white, circle, inner sep=2pt] (u) at (5, 0) {};
    \node[fill=black, circle, inner sep=2pt, label=left:$u_1$] (u0) at (4, -1) {};
    \node[draw=black, fill=white, circle, inner sep=2pt] (u1) at (6, -1) {};
    \node[fill=black, circle, inner sep=2pt] (u00) at (3.5, -2) {};
    \node[fill=black, circle, inner sep=2pt, label=left:$u_2$] (u01) at (4.5, -2) {};
    \node[draw=black, fill=white, circle, inner sep=2pt] (u10) at (5.5, -2) {};
    \node[draw=black, fill=white, circle, inner sep=2pt] (u11) at (6.5, -2) {};
    \node[fill=black, circle, inner sep=2pt] (u000) at (3.25, -3) {};
    \node[fill=black, circle, inner sep=2pt] (u001) at (3.75, -3) {};
    \node[draw=black, fill=white, circle, inner sep=2pt, label=below:$v_1$] (u010) at (4.25, -3) {};
    \node[draw=black, fill=white, circle, inner sep=2pt, label=below:$v_2$] (u011) at (4.75, -3) {};
    \node[draw=black, fill=white, circle, inner sep=2pt] (u100) at (5.25, -3) {};
    \node[draw=black, fill=white, circle, inner sep=2pt] (u101) at (5.75, -3) {};
    \node[draw=black, fill=white, circle, inner sep=2pt] (u110) at (6.25, -3) {};
    \node[draw=black, fill=white, circle, inner sep=2pt] (u111) at (6.75, -3) {};
    \draw (u) -- (u0);
    \draw (u) -- (u1);
    \draw (u0) -- (u00);
    \draw (u0) -- (u01);
    \draw (u1) -- (u10);
    \draw (u1) -- (u11);
    \draw (u00) -- (u000);
    \draw (u00) -- (u001);
    \draw (u01) -- (u010);
    \draw (u01) -- (u011);
    \draw (u10) -- (u100);
    \draw (u10) -- (u101);
    \draw (u11) -- (u110);
    \draw (u11) -- (u111);

    \draw[->] (7,-2) -- (8,-2);

    \node[draw=black, fill=white, circle, inner sep=2pt] (u) at (10, 0) {};
    \node[draw=black, fill=white, circle, inner sep=2pt] (u0) at (9, -1) {};
    \node[draw=black, fill=white, circle, inner sep=2pt] (u1) at (11, -1) {};
    \node[fill=black, circle, inner sep=2pt] (u00) at (8.5, -2) {};
    \node[draw=black, fill=white, circle, inner sep=2pt] (u01) at (9.5, -2) {};
    \node[draw=black, fill=white, circle, inner sep=2pt] (u10) at (10.5, -2) {};
    \node[draw=black, fill=white, circle, inner sep=2pt] (u11) at (11.5, -2) {};
    \node[fill=black, circle, inner sep=2pt] (u000) at (8.25, -3) {};
    \node[fill=black, circle, inner sep=2pt] (u001) at (8.75, -3) {};
    \node[fill=black, circle, inner sep=2pt] (u010) at (9.25, -3) {};
    \node[fill=black, circle, inner sep=2pt] (u011) at (9.75, -3) {};
    \node[draw=black, fill=white, circle, inner sep=2pt] (u100) at (10.25, -3) {};
    \node[draw=black, fill=white, circle, inner sep=2pt] (u101) at (10.75, -3) {};
    \node[draw=black, fill=white, circle, inner sep=2pt] (u110) at (11.25, -3) {};
    \node[draw=black, fill=white, circle, inner sep=2pt] (u111) at (11.75, -3) {};
    \draw (u) -- (u0);
    \draw (u) -- (u1);
    \draw (u0) -- (u00);
    \draw (u0) -- (u01);
    \draw (u1) -- (u10);
    \draw (u1) -- (u11);
    \draw (u00) -- (u000);
    \draw (u00) -- (u001);
    \draw (u01) -- (u010);
    \draw (u01) -- (u011);
    \draw (u10) -- (u100);
    \draw (u10) -- (u101);
    \draw (u11) -- (u110);
    \draw (u11) -- (u111);
\end{tikzpicture}
    \caption{The results of performing a $u_1$--$v_1$ exchange and a $u_2$--$v_2$ exchange on a subset of vertices $S$ (left), which yields a down-compressed set $S'$ (right)}
    \label{fig:enter-label}
\end{figure}

One advantage of considering down-compressed vertex subsets is that if $u\in S$, we can infer that many of $u$'s children and grandchildren are in $S$. 

\renewcommand{\thesubfigure}{\arabic{subfigure}}

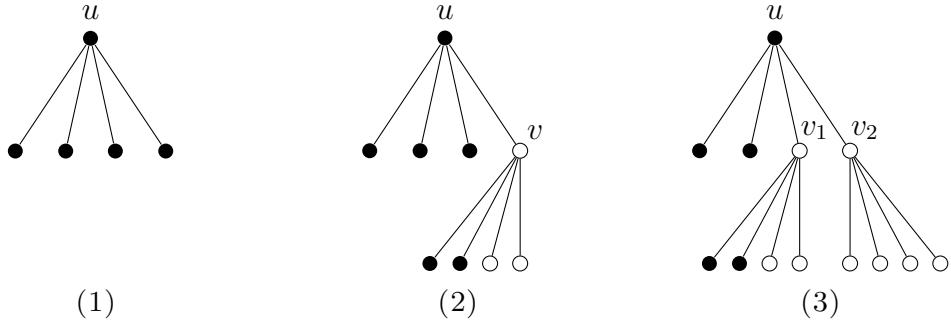
\begin{figure}
    \centering
\begin{subfigure}{0.3\textwidth}
\centering
    \begin{tikzpicture}
    \node[fill=black, circle, inner sep=2pt, label=above:$u$] (u) at (0, 0) {};
    \node[fill=black, circle, inner sep=2pt] (v1) at (-1, -1.5) {};
    \node[fill=black, circle, inner sep=2pt] (v2) at (-0.33, -1.5) {};
    \node[fill=black, circle, inner sep=2pt] (v3) at (0.33, -1.5) {};
    \node[fill=black, circle, inner sep=2pt] (v4) at (1, -1.5) {};
    \draw (u) -- (v1);
    \draw (u) -- (v2);
    \draw (u) -- (v3);
    \draw (u) -- (v4);
\end{tikzpicture}
\vspace{1.5cm}
\caption{}
\end{subfigure}
\begin{subfigure}{0.3\textwidth}
\centering
    \begin{tikzpicture}
    \node[fill=black, circle, inner sep=2pt, label=above:$u$] (u2) at (0, 0) {};
    \node[fill=black, circle, inner sep=2pt] (v1) at (-1, -1.5) {};
    \node[fill=black, circle, inner sep=2pt] (v2) at (-0.33, -1.5) {};
    \node[fill=black, circle, inner sep=2pt] (v3) at (0.33, -1.5) {};
    \node[draw=black, fill=white, circle, inner sep=2pt, label={[xshift=0.2cm, yshift=-0.10cm]$v$}] (v4) at (1, -1.5) {};
    \draw (u2) -- (v1);
    \draw (u2) -- (v2);
    \draw (u2) -- (v3);
    \draw (u2) -- (v4);
    
    \node[fill=black, circle, inner sep=2pt] (w1) at (-0.2, -3) {};
    \node[fill=black, circle, inner sep=2pt] (w2) at (0.2, -3) {};
    \node[draw=black, fill=white, circle, inner sep=2pt] (w3) at (0.6, -3) {};
    \node[draw=black, fill=white, circle, inner sep=2pt] (w4) at (1, -3) {};
    \draw (v4) -- (w1);
    \draw (v4) -- (w2);
    \draw (v4) -- (w3);
    \draw (v4) -- (w4);
    \end{tikzpicture}
    
\caption{}
\end{subfigure}
\begin{subfigure}{0.3\textwidth}
\centering
    \begin{tikzpicture}
    \node[fill=black, circle, inner sep=2pt, label=above:$u$] (u3) at (0, 0) {};
    \node[fill=black, circle, inner sep=2pt] (v1) at (-1, -1.5) {};
    \node[fill=black, circle, inner sep=2pt] (v2) at (-0.33, -1.5) {};
    \node[draw=black, fill=white, circle, inner sep=2pt, label={[xshift=0.2cm, yshift=-0.10cm]$v_1$}] (v3) at (0.33, -1.5) {};
    \node[draw=black, fill=white, circle, inner sep=2pt, label={[xshift=0.2cm, yshift=-0.10cm]$v_2$}] (v4) at (1, -1.5) {};
    
    \draw (u3) -- (v1);
    \draw (u3) -- (v2);
    \draw (u3) -- (v3);
    \draw (u3) -- (v4);
    
    \node[fill=black, circle, inner sep=2pt] (w1) at (-0.87, -3) {};
    \node[fill=black, circle, inner sep=2pt] (w2) at (-0.47, -3) {};
    \node[draw=black, fill=white, circle, inner sep=2pt] (w3) at (-0.07, -3) {};
    \node[draw=black, fill=white, circle, inner sep=2pt] (w4) at (0.33, -3) {};
    \draw (v3) -- (w1);
    \draw (v3) -- (w2);
    \draw (v3) -- (w3);
    \draw (v3) -- (w4);
    
    \node[draw=black, fill=white, circle, inner sep=2pt] (x1) at (1, -3) {};
    \node[draw=black, fill=white, circle, inner sep=2pt] (x2) at (1.4, -3) {};
    \node[draw=black, fill=white, circle, inner sep=2pt] (x3) at (1.8, -3) {};
    \node[draw=black, fill=white, circle, inner sep=2pt] (x4) at (2.2, -3) {};
    \draw (v4) -- (x1);
    \draw (v4) -- (x2);
    \draw (v4) -- (x3);
    \draw (v4) -- (x4);
    \end{tikzpicture}
    
\caption{}
\end{subfigure}
    \caption{Examples of the three possible configurations of a left- and down-compressed subset of vertices $S$, where filled vertices represent those vertices in $S$.}
    \label{fig:down_left_compressed}
\end{figure}

\begin{lemma} \label{lem:downward_compressed_initial}
    Suppose $T$ is a complete $q$-ary tree, and let $S\subseteq V(T)$ such that $S$ is left- and down-compressed. 
    For every vertex $u \in S$, either:
    \begin{enumerate}
        \item all children of $u$ are in $S$; 
        \item there is exactly one child of $u$ that is not in $S$, say $v$, and $v$ has at least one child not in $S$; or
        \item there is exactly two children of $u$ that are not in $S$, say $v_1$ and $v_2$, such that $v_1$ has at least one child in $S$ and at least one child not in $S$, and $v_2$ has no children in $S$.
    \end{enumerate}
\end{lemma}
\begin{proof}
We first observe that if the children of $u$ are leaves of $T$ and some child of $u$, say $v$, is not in $S$, then $S$ is not down-compressed. 
This follows since under these conditions $\delta((S\cup\{v\}) \setminus \{u\}) \subseteq (\delta(S) \cup \{u\})\setminus \{v\}$ as $v \in \delta(S)$ and $u\in S$, which yields  $|\delta((S\cup\{v\}) \setminus \{u\})|  \leq |\delta(S)|$ as $v \in \delta(S)$.
It follows that if the children of $u$ are leaves of $T$, then the first condition holds. 

See Figure~\ref{fig:down_left_compressed} for an example of each of the three possible cases given in the statement of this lemma. 
We may then suppose that $u$ has a distance of at least $2$ from a leaf and also that none of the three conditions of the theorem statement are met. 
    The vertex $u$ has either:
    \begin{enumerate}
    \item at least three children that are not in $S$. Due to being left-compressed, observe that either one of these children has all of its children in $S$, or two of these children have none of their children in $S$; 
    \item two children that are not in $S$, with either one of these children having all of its children in $S$, or both of these children having none of their children in $S$; or
    \item one child that is not in $S$, and this child has all of its children in $S$.  
    \end{enumerate}
    If we have a child of $u$, say $v$, that is not in $S$ and has all of its children in $S$, then $S$ is not down-compressed since $v$ is further from the root than $u$, but $|\delta((S \cup \{v\})\setminus \{u\})| \leq |\delta(S)|$. 
    Removing this possibility in each of the enumerated cases leaves us with $u$ having at least two children, say $v_1$ and $v_2$, such that neither $v_1$ nor $v_2$ has a child in $S$. 
    We will show that this is not possible. 
    Note that $v_1,v_2 \in \delta(S)$ and $v_1,v_2 \notin \delta(S\setminus\{u\})$. 
    If $\Desc(u)$ has a leaf that is not in $S$, label this leaf as $w$. 
    If not, then let $w\in \Desc(u)$ be the closest vertex to a leaf such that $w \notin  S$ and with all of $w$'s children in $S$.
    Observe that $\delta\big((S\cup \{w\})\setminus\{u\}\big) \subseteq \big(\delta(S) \setminus \{v_1,v_2\} \big)\cup \{u, \Pa(w)\}$, and so it follows that $|\delta((S\cup \{w\})\setminus\{u\})| \leq |\delta(S)|$, and so the $u$--$w$ swap is a compression. 
    Noting that $w$ is closer to a leaf than $u$, this contradicts the fact that $S$ was assumed to be down-compressed.
\end{proof}

Suppose $S$ is a left- and down-compressed subset of $V(T)$. 
By applying Lemma~\ref{lem:downward_compressed_initial}, we are able to show that if some vertex $u \in S$, then $S$ contains a vertex in $\Desc(u)$ in each level below $u$. 

\begin{lemma} \label{lem:downward_compressed_all_levels_compressed}
    Suppose $T$ is a complete $q$-ary tree of depth $d$, and let $S\subseteq V(T)$ such that $S$ is left- and down-compressed. 
    Suppose $u \in V(T)$ such that $\Desc(u) \cap S \cap L_i$ is non-empty. 
    We then have that $\Desc(u) \cap S \cap L_{i'}$ is non-empty for all $i \leq i' \leq d$. 
\end{lemma}
\begin{proof}
    This result follows from 
    Lemma~\ref{lem:downward_compressed_initial}, except in the case that $q=2$. 
    In this case, a problem occurs when $\Desc(u) \cap S \cap L_i$ contains a single vertex,
    $\Desc(u) \cap S \cap L_{i+1}$ is empty, 
    and $\Desc(u) \cap S \cap L_{i+2}$ contains a single vertex. 
    In this case, however, the assumption of $S$ being down-compressed is broken, as we could remove the vertex in $\Desc(u) \cap S \cap L_i$ from $S$ and place the left-most vertex of $\Desc(u) \cap L_{i+1}$, say $v$, into $S$. 
    We have $\delta\big((S\setminus \{u\})\cup \{v\}\big) \subseteq \big(\delta(S) \setminus\{v,u_2\}\big)\cup \{v_2,u\}$, with $u_2,v_2$ being the second children of $u$ and $v$, respectively. 
    It follows that $|\delta\big((S\setminus \{u\})\cup \{v\}\big)|\leq |\delta(S)|$, and so this $u$--$v$ exchange does not increase the border size, but it swaps a vertex in $S$ with another vertex not in $S$ that is closer to a leaf. 
    This contradicts the fact that $S$ is down-compressed. 
\end{proof}

\subsection{Aeolian compressions}

The process by which winds push sand to form sand dunes is known as an \emph{aeolian process}. 
After $S\subseteq V(T)$ is left- and down-compressed, we will show that $S$ can be compressed into various `peaks', where all descendants of $u$ (except, perhaps, $u$ itself) are prioritized over descendants of $v$ when $u$ is to the left of $v$.
Given $u,v\in V(T)$ with $u$ to the left of $v$, define an \emph{aeolian compression} to be any compression that removes vertices of $\Desc(v)$ from $S$ and places an equal number of vertices of $\Desc(u)$ into $S$. 
We will use aeolian compressions where we describe removing vertices from $\Desc(v')$ for some $v' \in \Desc(v)$ instead of removing the vertices from $\Desc(v)$, although the move will still be an aeolian compression. 
We define the function $\mathrm{aeolian}:\mathcal{P}(V(T)) \rightarrow \mathcal{P}(V(T))$ that performs the left compression function if $S$ is not left-compressed and the down compression if $S$ is not down-compressed. The function otherwise applies the aeolian compression that has $u$ and $v$ to be earliest in the ordering on the vertices of $T$, and which moves the maximum number of vertices from $\Desc{v}$ to $\Desc{u}$ in $S$.  
That is, $S$ is \emph{aeolian-compressed} if it is left-compressed, down-compressed, and if no aeolian compression of $S$ exists.  

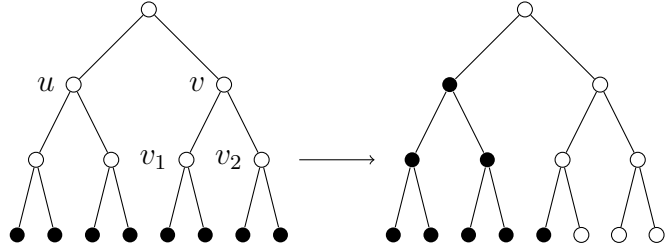
\begin{figure}
    \centering
    \begin{tikzpicture}
    \node[draw=black, fill=white, circle, inner sep=2pt] (u) at (5, 0) {};
    \node[draw=black, fill=white, circle, inner sep=2pt, label=left:$u$] (u0) at (4, -1) {};
    \node[draw=black, fill=white, circle, inner sep=2pt, label=left:$v$] (u1) at (6, -1) {};
    \node[draw=black, fill=white, circle, inner sep=2pt] (u00) at (3.5, -2) {};
    \node[draw=black, fill=white, circle, inner sep=2pt] (u01) at (4.5, -2) {};
    \node[draw=black, fill=white, circle, inner sep=2pt, label=left:$v_1$] (u10) at (5.5, -2) {};
    \node[draw=black, fill=white, circle, inner sep=2pt, label=left:$v_2$] (u11) at (6.5, -2) {};
    \node[fill=black, circle, inner sep=2pt] (u000) at (3.25, -3) {};
    \node[fill=black, circle, inner sep=2pt] (u001) at (3.75, -3) {};
    \node[fill=black, circle, inner sep=2pt] (u010) at (4.25, -3) {};
    \node[fill=black, circle, inner sep=2pt] (u011) at (4.75, -3) {};
    \node[fill=black, circle, inner sep=2pt] (u100) at (5.25, -3) {};
    \node[fill=black, circle, inner sep=2pt] (u101) at (5.75, -3) {};
    \node[fill=black, circle, inner sep=2pt] (u110) at (6.25, -3) {};
    \node[fill=black, circle, inner sep=2pt] (u111) at (6.75, -3) {};
    \draw (u) -- (u0);
    \draw (u) -- (u1);
    \draw (u0) -- (u00);
    \draw (u0) -- (u01);
    \draw (u1) -- (u10);
    \draw (u1) -- (u11);
    \draw (u00) -- (u000);
    \draw (u00) -- (u001);
    \draw (u01) -- (u010);
    \draw (u01) -- (u011);
    \draw (u10) -- (u100);
    \draw (u10) -- (u101);
    \draw (u11) -- (u110);
    \draw (u11) -- (u111);

    \draw[->] (7,-2) -- (8,-2);

    \node[draw=black, fill=white, circle, inner sep=2pt] (u) at (10, 0) {};
    \node[fill=black, circle, inner sep=2pt] (u0) at (9, -1) {};
    \node[draw=black, fill=white, circle, inner sep=2pt] (u1) at (11, -1) {};
    \node[fill=black, circle, inner sep=2pt] (u00) at (8.5, -2) {};
    \node[fill=black, circle, inner sep=2pt] (u01) at (9.5, -2) {};
    \node[draw=black, fill=white, circle, inner sep=2pt] (u10) at (10.5, -2) {};
    \node[draw=black, fill=white, circle, inner sep=2pt] (u11) at (11.5, -2) {};
    \node[fill=black, circle, inner sep=2pt] (u000) at (8.25, -3) {};
    \node[fill=black, circle, inner sep=2pt] (u001) at (8.75, -3) {};
    \node[fill=black, circle, inner sep=2pt] (u010) at (9.25, -3) {};
    \node[fill=black, circle, inner sep=2pt] (u011) at (9.75, -3) {};
    \node[fill=black, circle, inner sep=2pt] (u100) at (10.25, -3) {};
    \node[draw=black, fill=white, circle, inner sep=2pt] (u101) at (10.75, -3) {};
    \node[draw=black, fill=white, circle, inner sep=2pt] (u110) at (11.25, -3) {};
    \node[draw=black, fill=white, circle, inner sep=2pt] (u111) at (11.75, -3) {};
    \draw (u) -- (u0);
    \draw (u) -- (u1);
    \draw (u0) -- (u00);
    \draw (u0) -- (u01);
    \draw (u1) -- (u10);
    \draw (u1) -- (u11);
    \draw (u00) -- (u000);
    \draw (u00) -- (u001);
    \draw (u01) -- (u010);
    \draw (u01) -- (u011);
    \draw (u10) -- (u100);
    \draw (u10) -- (u101);
    \draw (u11) -- (u110);
    \draw (u11) -- (u111);
\end{tikzpicture}
    \caption{The results of performing the aeolian compression moving two vertices from $\Desc(v_2)$ to $\Desc(u)$, and then the aeolian compression moving one vertex from $\Desc(v_1)$ to $\Desc(u)$ on a subset of vertices $S$ (left), which results in an aeolian-compressed set $S'$ (right).}
    \label{fig:enter-label}
\end{figure}
We have defined aeolian-compressed so that if some $S$ is not aeolian-compressed, then there exists some compression that can be applied that would bring us closer to being aeolian compressed, either a treeswap, a $u$--$v$ exchange, or an aeolian compression. 
As such, we have a natural corollary that follows from the definitions. 

\begin{corollary} \label{cor:aeolian}
    Let $T$ be a complete $q$-ary tree. 
    For each $s$, there exists an aeolian-compressed set $S$ with $|S|=s$ and $\Phi_V(T,s) = |\delta(S)|$.
\end{corollary}

While Corollary~\ref{cor:aeolian} is a simple observation, the border size $|\delta(S)|$ of an aeolian-compressed set $S$ is relatively easy to analyze, and so we may use this to find bounds on $\Phi_V(T,s)$. 
For an aeolian-compressed set $S$ and $u,v\in V(T)$ where $u$ is to the left of $v$, we may analyze how $S$ intersects with $\Desc(u)$ and $\Desc(v)$.  
We show that if any vertex of $S$ is in $\Desc(v)$, then it must follow that all vertices of $\Desc(u)\setminus\{u\}$ are in $S$. 

\begin{lemma} \label{lem:uv_u_leaves_filled}
    Suppose $T$ is a complete $q$-ary tree of depth $d$, and let $S \subseteq V(T)$ such that $S$ is aeolian-compressed. 
    Suppose that $u$ and $v$ are two vertices on the same level in $T$, that $u$ is to the left of $v$, and suppose that both $\Desc(u)$ and $\Desc(v)$ intersect with $S$. 
    We then have that either $\Desc(u) \setminus S=\emptyset$ or $\Desc(u) \setminus S=\{u\}$.
\end{lemma}
\begin{proof}
    By Lemma~\ref{lem:downward_compressed_all_levels_compressed}, $\Desc(v)\cap S \cap L_d$ is non-empty. 
    As $S$ is left-compressed, all vertices in $\Desc(u) \cap L_d$ must be in $S$, and so     
    \(
    \Desc(u)\cap S \cap L_d = \Desc(u) \cap L_d.
    \)

    If we assume, for the sake of contradiction, that $|\Desc(u) \setminus S|\geq 2$, then it follows from the above argument that there is some level $i$ where \(\Desc(u)\cap S \cap L_{i} \neq \Desc(u) \cap L_{i}\) and such that for all $i'>i$, \(\Desc(u)\cap S \cap L_{i'} = \Desc(u) \cap L_{i'}\). 
    Let $x$ be the left-most vertex in $\Desc(u) \cap L_{i-1}$ such that $\children(x)\cap S \neq \children(x)$. 
    Let $X=\children(x) \setminus S$. 

We will make observations about $\Desc(v)$, although this will break into two cases. 
    In the first case, suppose that no $j\geq i$ has $\Desc(v)\cap S\cap L_j =  \emptyset$ and $\Desc(v)\cap S\cap L_{j-1} = \emptyset$. 
    This gives that $x\in S$, as otherwise $\Desc(v)\cap S\cap L_i = \emptyset$ and $\Desc(v)\cap S\cap L_{i-1} = \emptyset$, since $S$ is left-compressed. 
    However, this implies that the $u'$--$v'$ swap $S'$ between a $u' \in X$ and $v' \in \Desc(v)\cap S$ is a compression, as then $\delta(S') \subseteq (\delta(S)\setminus \{u'\} )\cup \{v'\}$ with $u' \in \delta(S)$, and so $|\delta(S')| \leq |\delta(S)|$. 
    This contradicts the fact that $S$ is aeolian-compressed. 
    
    Thus, we may assume that the second case occurs, which is that there is a $j\geq i$ such that $\Desc(v)\cap S\cap L_j = \emptyset$ and $\Desc(v)\cap S\cap L_{j-1} = \emptyset$. 
    Assume that $j$ is as large as possible, but noting that $j<d$, or else $S$ would not be down-compressed. 
    Let $y$ be the right-most vertex in $\Desc(v)\cap L_{j}$ such that $\children(y) \cap S \neq \emptyset$. 
    Let $Y = \children(y) \cap S$. 

    Consider the following three cases, each with a corresponding aeolian compression:
    \begin{enumerate}
        \item If $|X|=|Y|$, remove $Y$ from $S$ and add $X$ to $S$.  
        \item If $|X|>|Y|$, remove $Y$ from $S$ and add $|Y|$ vertices in $X$ to $S$ (call this set of $|Y|$ vertices $X'$).  
        \item If $|X|<|Y|$, remove $|X|$ vertices of $Y$ (call this set of $|X|$ vertices $Y'$) from $S$ and add $X$ to $S$. 
    \end{enumerate}
    In all three subcases, the resulting set $S'$ has the same number of vertices as $S$. 

    We will show that each of these three operations are compressions, and hence, aeolian compressions. 
    Each child of a vertex in $X$ is in $S$ by the definition of $i$, and at least one child of each vertex in $Y$ is in $S$ by Lemma~\ref{lem:downward_compressed_all_levels_compressed}. 
    For case $1$, we have that $\delta(S') = (\delta(S) \cup Y)\setminus (X\cup \{y\})$, except in the case that $|X|=|Y|=q$, where $\delta(S') = (\delta(S) \cup Y \cup \{x\})\setminus ( X \cup \{y\})$.
    In both of these cases, the border size has not changed. 
    For case $2$, $\delta(S') = \delta(S) \cup Y \setminus (X' \cup \{y\})$, and so the border size has decreased by $1$. 
    For case $3$, $\delta(S') = \delta(S) \cup Y' \setminus X$, and so the border size has not changed. 
    
    Therefore, in each possible case, an aeolian compression was found, and so by contradiction, it must be that $|\Desc(u) \setminus S|\leq 1$, as required. 
    In the case that $|\Desc(u) \setminus S| = 1$, suppose for contradiction that $\Desc(u) \setminus S = \{v\}$ for some $v \neq u$. 
    The $u$--$v$ exchange is a compression, contradicting the fact that $S$ is down-compressed.
    Therefore, $\Desc(u) \setminus S= \emptyset$ or $\Desc(u) \setminus S= \{u\}$, as required. 
\end{proof}

\section{Isoperimetric Peak Results}
\subsection{Lower bounds}
Lemma~\ref{lem:uv_u_leaves_filled} shows that any aeolian-compressed vertex subset has an incredibly restricted structure. 
Since any vertex subset can be compressed to be an aeolian-compressed vertex subset with an equal or smaller-sized vertex border, it is now possible to analyze the minimum size of the vertex border of all vertex subsets of a given cardinality. 
Recall that the $n$th child of vertex $u$ is denoted $u^{(n)}$. 
Also, take $t_i$ to be the number of vertices in a complete $q$-ary tree of depth $d-i$; that is: 
    \[t_i = \frac{q^{d-i+1}-1}{q-1}.\]

We will briefly describe the general mechanism of the proof when $q \geq 5$; the proofs when $q<5$ will modify this mechanism.  
Let $S$ be an aeolian-compressed vertex subset of some complete $q$-ary tree, let $u_i$ be a vertex of the tree on level $i$, and suppose $S$ and $u_i$ are defined such that $\Desc(u_i)\cap S$ contains between $2 t_{i+1}+1$ and $3(t_{i+1}-1)-1$ vertices. 
Using Lemma~\ref{lem:uv_u_leaves_filled}, each vertex in $\Desc(u_i^{(j)})\setminus \{u_i^{(j)}\}$ must be in $S$ for $j \in  \{ 1,2\}$, and no vertex in $\Desc(u_i^{(j)})$ can be in $S$ for $j \geq 4$.
Additionally, we will see that the vertices $\{u_i^{(1)},u_i^{(2)}\}$ must also be in $S$ to prevent the vertex border from being too large. 
As $q\geq 5$, if the vertex $u_i$ was itself included in $S$, then the border would increase in size, and so $u_i \in \delta(S)$.
For a recursion, we find $\Desc(u_{i+1}) = \Desc(u_i^{(3)})$ contains from $2t_{i+2}+1$ to $3(t_{i+2}-1)-1$ vertices in $S$. 
As a consequence, each level except level $d$ must contain at least one vertex in the vertex border. 
It follows that the vertex border will have a cardinality of at least $d$. 

\begin{theorem} \label{thm:lower_q5}
    For $T$ a complete $q$-ary tree of depth $d$ with $q\geq 5$, $\Phi_V(T) \geq d.$
\end{theorem}
\begin{proof}
    Define $T_i = \sum_{i'=1}^{i} 2t_{i'}$ and the corresponding $\bar{T}_i = \sum_{i'=i+1}^{d} 2t_{i'}$ for each $i \in [0,d]$, where we note that $\bar{T}_0 = T_i + \bar{T}_i$ for each $i \in [0,d]$.
    Suppose that $S$ is an aeolian-compressed subset of $V(T)$ with $|S|=\bar{T}_0$ and $|\delta(S)| = \Phi_V(T,\bar{T}_0)$. 
    We will show that $|\delta(S)|\geq d$, which implies that $\Phi_V(T) \geq  \Phi_V(T,\bar{T}_0) \geq d$, as required.

    This proof will show that various subsets of $V(T)$ must be in $S$ using a level-by-level argument, starting with $i=0$ and going through to $i=d$. 
    Let $u_0, \ldots, u_d$ be the root-to-leaf path such that $u_{i+1}= u_i^{(3)}$, where $u_0$ is the root vertex.
    Let $b_i$ be the count of the number of times that 
    \(
        \{u_{i'}^{(1)},u_{i'}^{(2)}\}
        \subseteq 
        S
    \)
    for $i'$ with $0 \leq i' < i$. 
    Let $\delta_i = |\delta(S) \setminus \Desc(u_i)|$. 
    
    Assume the inductive hypothesis that, for a given $i\geq 0$, first that 
    \[
    \bigcup_{0 \leq i'<i}
        \big(
            \Desc(u_{i'}^{(1)})
                \cup 
            \Desc(u_{i'}^{(2)})
        \big)
        \setminus
        \{u_{i'}^{(1)},u_{i'}^{(2)}\}
    \subseteq 
    S
    ,\]
     second that
    \(\delta_i = 2i-b_i\), and third that 
    \(
        \bar{T}_i 
        \leq 
            |S \cap 
            \Desc(u_i)|
            < 
            \bar{T}_i
                +2i-2b_i.
    \)
    The inductive hypothesis clearly holds for $i=0$. 

    We will now show that each of these properties holds for $i+1$ as long as $i+1 <d$.
    By the inductive hypothesis, the number of vertices in $S$ that are also in $\Desc(u_i)$ is greater than $\overline{T}_i \geq 2t_{i+1} + 2t_d = 2t_{i+1}+2$, since $i+1<d$. 
    We then have that some vertex in 
    \(
        \Desc(u_i)
        \setminus
        \big(
            \Desc(u_i^{(1)})\cup \Desc(u_i^{(2)})\cup \{u_i\}
        \big)
    \)
    must be in $S$, since this later set contains exactly $2t_{i+1}+1$ vertices. 
    Using this fact, Lemma~\ref{lem:uv_u_leaves_filled} yields that $\Desc(u_i^{(1)})\cup \Desc(u_i^{(2)}) \setminus \{u_i^{(1)},u_i^{(2)}\}\subseteq S$. 
    Combining with information from the inductive assumption, we have that 
    \[\bigcup_{0 \leq i'<i+1}
    \Desc(u_{i'}^{(1)})\cup \Desc(u_{i'}^{(2)})
    \setminus\{u_{i'}^{(1)},u_{i'}^{(2)}\}
    \subseteq S. \]
    Hence, we have shown the first property required for the induction. 

    We now focus on the vertices in 
    \(\big(  \Desc(v_i^{(4)}) \cup \Desc(v_i^{(5)})    \big)
        \cap S.
    \)
    Suppose, for the sake of contradiction, that this set is not empty. 
    Lemma~\ref{lem:uv_u_leaves_filled} yields that $\Desc(u_i^{(1)})\cup \Desc(u_i^{(2)})\cup  \Desc(u_i^{(3)})\setminus \{u_i^{(1)},u_i^{(2)},u_i^{(3)}\}\subseteq S$. 
    This implies that there must be at least $(3t_{i+1}-3)+1$ vertices in $\Desc(u_i)\cap S$.
    However, by the third assumption made for induction, this gives that $3t_{i+1}-2<\bar{T}_i+2i - 2b_i = 2t_{i+1} + \bar{T}_{i+1}+2i - 2b_i$, which rearranges as $t_{i+1}<\bar{T}_{i+1}+2i - 2b_i+2= \bar{T}_{i+1}+2\delta_i-2i+2\leq \bar{T}_{i+1}+2(d-i)$, as we may assume that $\delta_i \leq d-1$ since otherwise the proof would be finished. 
    To see why this is not possible, we now prove that 
    $t_{i+1}\geq  \bar{T}_{i+1}+2(d-i)$ using induction. 
    Observe that $t_{i+1}= 1+qt_{i+2} \geq 5t_{i+2} \geq 4t_{i+2} + \overline{T}_{i+2}$, where this last inequality can be seen from the definition of $\overline{T}_{i+1}$ by noting that $\overline{T}_{i+2} \leq 3 t_{i+1} \leq t_{i+2}$. 
    Continuing this calculation, we find that $4t_{i+2} + \overline{T}_{i+2} = 2t_{i+2} + \overline{T}_{i+1} \geq \overline{T}_{i+1} + 2(d-i)$, where this last inequality holds as $t_{i+2}$ counts the number of vertices in a complete $q$-ary tree of depth $d-i-2$, which contains $d-i-2$ vertices in a root-to-leaf path and at least $2$ addition vertices adjacent to the root vertex, for a total of at least $d-i$ vertices. 
    This shows that $t_{i+1} \geq \overline{T}_{i+1}+2(d-i)$. 
    We therefore get the contradiction required, and so \(\big(  \Desc(v_i^{(4)}) \cup \Desc(v_i^{(5)})    \big)
        \cap S
    \)
    is empty. 

It follows that $u_{i} \notin S$, or else Lemma~\ref{lem:downward_compressed_initial} would be contradicted, as we have shown that $u_{i}$ has two children not in $S$, both of which have no children in $S$. We also have that 
\begin{enumerate}
    \item $u_{i}^{(1)},u_{i}^{(2)}\in \delta(S)$ if $u_{i}^{(1)},u_{i}^{(2)} \notin S$; 
    \item $u_{i}^{(2)},u_{i}\in \delta(S)$ if $u_{i}^{(1)} \in S$ and $u_{i}^{(2)}\notin S$;
    \item $u_{i}^{(1)},u_{i}\in \delta(S)$ if $u_{i}^{(2)} \in S$ and $u_{i}^{(1)}\notin S$; and 
    \item $u_{i}\in \delta(S)$ if $u_{i}^{(1)}, u_{i}^{(2)} \in S$. 
\end{enumerate}
In case $4$, we have that $b_{i+1}=b_i+1$ and $\delta_{i+1}= \delta_{i}+1 = (2i+2)-(b_i+1) = 2(i+1)-b_{i+1}$. 
In the other cases, $b_{i+1} = b_i$ and $\delta_{i+1}= \delta_{i}+2 = (2i+2)-b_i = 2(i+1)-b_{i+1}$.
We have thus shown the second assumption of the inductive hypothesis. 

We note that $|S \cap \Desc(u_i)|= |S \cap \Desc(u_{i+1})|+ 2t_{i+1} -2 + \varepsilon$, where $\varepsilon$ is $0$, $1$, $1$, or $2$ in C
Cases $1$, $2$, $3$, and $4$ above, respectively, noting that $u_{i}\in S$ implies $u_{i}^{(1)},u_{i}^{(2)} \in S$ since $S$ is aeolian-compressed. 
Using this equation in the inductive assumption that $\bar{T}_i \leq |S \cap \Desc(u_i)|< \bar{T}_i+2i-2b_i$ yields that 
$\bar{T}_{i+1} \leq |S \cap \Desc(u_{i+1})| + \varepsilon-2< \bar{T}_{i+1} +2i-2b_i$. 
Since $\epsilon \leq 2$, it follows that $\bar{T}_{i+1} \leq |S \cap \Desc(u_{i+1})| + \varepsilon-2 \leq |\Desc(u_{i+1})|$, and so the left-hand inequality has been shown. 
In Cases $1$, $2$, and $3$ from above, we also have $|\Desc(u_{i+1})| + \varepsilon-2< \bar{T}_{i+1} +2i-2b_i$ implies $|\Desc(u_{i+1})| < \bar{T}_{i+1} +2(i+1)-2b_{i+1}$, since $b_{i+1}=b_{i}$ in these cases. 
In the remaining Case $4$, we have that
$|\Desc(u_{i+1})| + \varepsilon-2< \bar{T}_{i+1} +2i-2b_i$ implies $|\Desc(u_{i+1})| < \bar{T}_{i+1} +2(i+1)-2b_{i}-\varepsilon = \bar{T}_{i+1} +2(i+1)-2b_{i+1}$, since $b_{i+1}=b_i+1$ and $\varepsilon=2$. 
We therefore have that $\bar{T}_{i+1} 
\leq
|\Desc(u_{i+1})| 
<  \bar{T}_{i+1} +2(i+1)-2b_{i+1}$, which is the third and final assumption of the inductive hypothesis. 

The inductive hypothesis clearly holds for the base case, $i=0$. 
We may recursively apply the induction until $i+1=d$, at which point the properties no longer hold, and then apply the argument again. 
In this case, however, we may assume that $\delta_{d-1} = 2(d-1)-b_{d-1} \leq d-1$, otherwise $\delta_{d-1} \geq d$, and the proof would be finished.
This then yields $b_{d-1}\geq d-1$. 
Since $b_{d-1}\leq d-1$ as $b_i$ counts the number of layers that satisfy a property, this implies that $b_{d-1} = d-1$, and we have that $\{u_{i'}^{(1)},u_{i'}^{(2)}\}\subseteq S$ for $0 \leq i' < d-1$. 
As a consequence, $S \setminus \Desc(u_{d-1}) = \bigcup_{i'=0}^{d-2} \Desc(u_{i'}^{(1)})\cup \Desc(u_{i'}^{(2)})$, and so 
$S \setminus \Desc(u_{d-1})$ has cardinality $T_{d-1}$ and does not contain the vertex $u_{d-2}$. 
We also have that $\delta_{d-1}=d-1$. 
We then know that there are $2t_d = 2$ vertices in $S \cap \Desc(u_{d-1})$, and these may only be the vertices $\{v_{d-1}^{(1)},v_{d-1}^{(2)}\}$ as $S$ is aeolian-compressed.  
This now gives that $\delta(S) = \delta_d = d$, and the proof is complete. 
\end{proof}

In Theorem~\ref{thm:lower_q5}, the fact that at most three of the children of a $u_i$ could be in $S$ meant that adding $u_i$ to $S$ would increase the vertex border size. 
When $q \leq 4$, this is no longer the case in general. 
We can modify our approach when $q \in \{3,4\}$ by 
instead supposing the subgraph $\Desc(u_i)$ contains between $t_{i+1}$ and $2(t_{i+1}-1)$ vertices in $S$ for each $u_i$ of the root-to-leaf path. 
Working under this assumption, we do not know whether $u_i$ or $u_i^{(1)}$ is in $S$ or not, as these can be included without increasing the vertex border. 
As such, there is enough flexibility to ensure that for some $u_i$ on the root-to-leaf path with $u_i$ having a distance of at most $\log_q(2d)$ to a leaf, the set of vertices $\Desc(u_i)\cap S$ is empty, as otherwise, we could remove such a vertex from $S$ and instead include some $u_i$ or $u_i^{(1)}$. 
There will be one vertex in the border for each level, from level $0$ through to level $d-\lceil\log_q(2d)\rceil$, yielding the result.

\begin{theorem} \label{thm:lower_q34}
    For $T$ a complete $q$-ary tree of depth $d$ with $q \in \{3,4\}$, 
    \[
    \Phi_V(T) \geq d - \log_q(d) - \big(\log_q(2) +1\big).
    \]

\end{theorem}
\begin{proof}
    Define $D = d - \lceil \log_q(2d)\rceil$. 
    Let $S \subseteq V(T)$ with $|S| = \sum_{i=1}^{D} (t_i 
 -1)$  such that  $|\delta(S)| = \Phi_V(T,|S|)$. 
    By Corollary~\ref{cor:aeolian}, we may assume without loss of generality that $S$ is aeolian-compressed. 
    Let $u_0, u_1,\ldots, u_d$ be the root-to-leaf path in $T$ with $u_{i+1} = u_i^{(2)}$ for $0 \leq i \leq d-1$.  

    Suppose for the sake of contradiction that there is some $I$ with $\Desc(u_I^{(3)}) \cap S \neq \emptyset$. 
    This implies $\Desc(u_{i})$ intersects with $S$ for all $1 \leq i \leq I$, since $u_I^{(3)} \in \Desc(u_{i})$. 
    Since $u_{i-1}^{(1)}$ is to the left of $u_{i}$, 
    it follows from Lemma~\ref{lem:uv_u_leaves_filled} for each $1 \leq i \leq I$ that 
    $\Desc(u_{i-1}^{(1)})\setminus \{u_{i-1}^{(1)}\} \subseteq S$, 
    and so 
    $|\Desc(u_{i-1}^{(1)})\cap S| \geq |\Desc(u_{i-1}^{(1)})|-1 = t_{i}-1$. 
    Similarly, $|\Desc(u_{I}^{(1)})\cap S| \geq t_{I-1}-1$ and $|\Desc(u_{I}^{(2)})\cap S| \geq t_{I-1}-1$ 
    since $u_I^{(1)}$ and $u_I^{(2)}$ are to the left of $u_I^{(3)}$ and $\Desc(u_I^{(3)}) \cap S \neq \emptyset$. 
    Since each of the $\Desc(u_{i}^{(1)})$ are disjoint from each other and from $\Desc(u_I^{(2)})$, this implies that $S$ has cardinality at least $(t_{I+1}-1) + \sum_{i=1}^{I+1} (t_{i}-1) > \sum_{i=1}^{D} (t_i -1)  = |S|$, as evidently, $(t_{I+1}-1)  > \sum_{i=I+2}^{D} (t_i -1)$. This gives a contradiction. 

    Now, suppose for the sake of contradiction that there is some $I<D$ such that $\Desc(u_I) \cap S = \emptyset$. 
    The set $S$ must be a subset of the vertices $\{u_0, \ldots, u_{I-1}\} \cup \bigcup_{i=0}^{I-1}\Desc(u_i^{(1)})$, since this is the set of vertices possible in $S$ when we exclude the vertices in $\Desc(u_I)$ and in $\Desc(u_i^{(3)}) \cap S \neq \emptyset$ for all $i$.
    The number of vertices in  $\{u_0, \ldots, u_{I-1}\} \cup \bigcup_{i=0}^{I-1}\Desc(u_i^{(1)})$ is 
\begin{eqnarray*}
I + \sum_{i=1}^I t_i  &= & 2I + \sum_{i=1}^I (t_i-1) \\
    &\leq & 2I + |S| - (t_{D}-1) \\
    &\leq & 2I +|S| - \frac{q^{\log_q{2d} +1}-1}{q-1}+1 \\
   &\leq & 2I +|S| - \frac{2dq-1}{q-1}+1 \\
    &\leq & |S| +2I - 2d+1 \\
    &\leq & |S|-1.
\end{eqnarray*}
    However, we then have that $|S| \leq |S|-1,$ and so we have the required contradiction. 

    Therefore, we must have $\Desc(u_j) \cap S \neq \emptyset$ for all $j$ with $0 \leq j < D$. 
    It then follows that 
    $\Desc(u_{j-1}^{(1)})\setminus \{u_{j-1}^{(1)}\} \subseteq S$ for $1 \leq j \leq D-1$, since $u_{j-1}^{(1)}$ is to the left of $u_{j}$. 
    The set of vertices $S' = \bigcup_{j=1}^{D-1} \Desc(u_{j-1}^{(1)}) \setminus \{u_{j-1}\}^{(1)} \subseteq S$ consists of $|S| -  (t_{D} -1)$ vertices.
    If $\Desc(u_{D-1}^{(2)})$ intersected with $S$, then $\Desc(u_{D}^{(1)})\setminus \{u_{D}^{(1)}\} \subseteq S$, but this implies that there would be an addition set of $(t_{D} -1)$ vertices of $S$ in $\Desc(u_{D}^{(1)})$ and at least one additional vertex of $S$ in $\Desc(u_{D-1}^{(2)})$, meaning $S$ would contain at least $|S|+1$ vertices. This is a contradiction. 
    Therefore, $\Desc(u_{D-1}^{(2)})$ does not intersect with $S$. 
    As a consequence, the $t_D-1$ vertices in $S$ that are not in $S'$ must occur in the set of vertices $\bigcup_{i=0}^{D-1}\{u_i,u_i^{(1)}\} \cup \big( \Desc(u_{D-1}^{(1)}) \setminus u_{D-1}^{(1)}\big)$. 
    The set of vertices $\bigcup_{i=0}^{D-1}\{u_i,u_i^{(1)}\}$ contains $2D\leq 2d-2$ vertices, and notice that 
\begin{eqnarray*}
    t_D-1 &=& \frac{q^{\lceil \log{2d}\rceil +1}-1}{q-1} \\
    &\geq & \frac{2dq-1}{q-1} \\
    &=& (2d-1)\frac{q}{q-1} \\ 
    &\geq& 2d-1.
\end{eqnarray*}    
    Therefore, $\Desc(u_{D-1}^{(1)}) \setminus u_{D-1}^{(1)}$ must contain at least one vertex also in $S$. 

    We know that $\Desc(u_{D-1}^{(1)})$ contains at most $t_D-1$ vertices, and so $\Desc(u_{D-1}^{(1)})$ must contain some vertex not in $S$ which is adjacent to a vertex in $S$, which is in $\delta(S)$. 
    For each $j$ with $1 \leq j \leq D-1$, either $u_{j-1}^{(1)}\in \delta(S)$ or $u_{j-1}^{(1)}\in S$. In the later case, either $u_{j-1}\in \delta(S)$ or $u_{j-1}\in S$. Again in this later case, then $u_{j-1}^{(3)}\in \delta(S)$. It follows that one of $\{u_{j-1}^{(1)}, u_{j-1}, u_{j-1}^{(3)}\}$ must be in $\delta(S)$ for each $j$. 
    It then follows that $\delta(S)$ contains at least $D$ distinct vertices. 
    Noting that $d - \lceil \log_q(2d)\rceil \geq d - \log_q(d) - \big(\log_q(2) +1\big)$, we have finished the proof. 
\end{proof}

In the case that $q=2$, we need to have even fewer vertices in the subset of vertices to ensure that the border is large. 
For example, if we had $$\Desc(u_i^{(1)}) \cup \Desc(u_{i+1}^{(1)}) \cup \ldots \cup \Desc(u_{i+j}^{(1)}) \cup \{ u_i, u_{i+1}, \ldots, u_{i+j}\} \subseteq S,$$ then there are at most two border vertices that are next to a vertex in this set, where we would have $j+1$ border vertices if $q>2$. 
We can circumvent this phenomenon by choosing only one descendant subtree to be in the aeolian-compressed $S$ for every two levels of the tree. 
If $S$ is aeolian-compressed, every second level will have an associated border vertex, resulting in a large border.

\begin{theorem} \label{thm:lower_q2}
    For $T$, a complete binary tree of depth $d$, 
    \[
    \Phi_V(T) \geq 
    \Big\lfloor \frac{d-\lceil \log_2(3d) \rceil}{2} \Big\rfloor
    \geq \frac{d - \log_2{d}}{2} -\frac{3+\log_2{3}}{2}.
    \] 
\end{theorem}
\begin{proof}
    Let $D = \lfloor (d-\lceil \log_2(3d) \rceil)/2 \rfloor$. 
    Note for later that this implies $t_{2D} \geq 3D+3$. 
    Let $S \subseteq V(T)$ with $|S| = \sum_{i'=1}^{D} (t_{2i'} -1)$  
    such that  $|\delta(S)| = \Phi_V(T,|S|)$. 
    By Corollary~\ref{cor:aeolian}, we may assume without loss of generality that $S$ is aeolian-compressed. 
     Let $u_0, u_1,\ldots, u_d$ be the root-to-leaf path in $T$ with $u_{2i+1} = u_{2i}^{(1)}$ and $u_{2i+2} = u_{2i+1}^{(2)}$.  
    This proof will show that various subsets of $V(T)$ must be in $S$ using an argument on two levels, $2i$ and $2i+1$, per step of the argument, starting with $i=0$ and going through to $i=D$. 

To set up an induction, define the property $P(i)$ to be true if the following three properties hold:  
\begin{equation}
\Big( \bigcup_{0 \leq i' < i} \Desc(u_{2i'}^{(2)})\Big) \cap S = \emptyset, 
\end{equation}
\begin{equation}
 \bigcup_{0 \leq i' < i} \Desc(u_{2i'+1}^{(1)})\setminus \{u_{2i'+1}^{(1)}\} \subseteq S, \text{ and} 
\end{equation}
\begin{equation}
    \sum_{i'=i+1}^{D} (t_{2i'}-1) -3i 
    \leq |\Desc(u_{2i})\cap S|
    \leq \sum_{i'=i+1}^{D} (t_{2i'}-1).
\end{equation}

The base case $P(0)$ is clearly true. 
Assume $P(i)$ is true for some $i$ with $0 \leq i \leq D-2$. 
We will show that $P(i+1)$ is true. 

If $\Desc(u_{2i}^{(2)}) \cap S \neq \emptyset$, then Lemma~\ref{lem:uv_u_leaves_filled} yields $\Desc(u_{2i}^{(1)}) \setminus \{u_{2i}^{(1)}\}  \subseteq S$, however $\Desc(u_{2i}^{(1)}) \setminus \{u_{2i}^{(1)}\}$ contains $t_{2i}>\sum_{i'=i+1}^{D} (t_{2i'}-1)$ vertices, contradicting property (3) of $P(i)$. 
Hence, $\Desc(u_{2i}^{(2)}) \cap S = \emptyset$, and property (1) of $P(i+1)$ holds. 
It follows that $\big(\Desc(u_{2i+1}) \cup \{u_{2i}\}\big) \cap S$ contains between $\sum_{i'=i+1}^{D} (t_{2i'}-1) -3i$ and $\sum_{i'=i+1}^{D} (t_{2i'}-1)$ vertices. 
Since $\Desc(u_{2i+1}^{(1)})\cup \{u_{2i}, u_{2i+1}\}$ contains $t_{2i+2}+2<\sum_{i'=i+1}^{D} (t_{2i'}-1) -3i$ vertices, 
$\Desc(u_{2i+1}^{(2)})\cap S = \Big(\big( \Desc(u_{2i+1}) \cup \{u_{2i}\}\big) \setminus \big(\Desc(u_{2i+1}^{(1)}) \cup \{ u_{2i}, u_{2i+1}\}\big)\Big)\cap S$ must contain at lest one vertex. 
By Lemma~\ref{lem:uv_u_leaves_filled}, $\Desc(u_{2i+1}^{(1)})\setminus \{u_{2i+1}^{(1)}\} \subseteq S$, and property (2) of $P(i+1)$ holds. 
There are $(t_{2i+2}-1)$ vertices in $\Desc(u_{2i+1}^{(1)}) \setminus \{u_{2i+1}^{(1)}\}$, which are all in $S$, and the vertices $\{u_{2i+1}^{(1)}, u_{2i+1}, u_{2i}\}$ may or may not be in $S$. 
All other vertices of $Desc(u_{2i+2})$ are not in $S$. 
It follows that there are between $\big(\sum_{i'=i+1}^{D} (t_{2i'}-1) -3i\big) - (t_{2i+2}-1)-3 = \sum_{i'=i+2}^{D}(t_{2i'}-1) -3(i+1)$ and 
$\sum_{i'=i+1}^{D} (t_{2i'}-1)  - (t_{2i+2}-1) = \sum_{i'=i+2}^{D}(t_{2i'}-1)$ vertices in $\Desc(u_{2i+2})\cap S$, and so property (3) of $P(i+1)$ holds. 

The induction then holds, and so we know that $P(D-1)$ holds. 
From condition (2) of $P(D-1)$, it follows that for each $i'$ with $0 \leq i' <D-1$, either $u_{2i'+1}^{(1)} \in\delta(S)$ or $u_{2i'+1}^{(1)} \in S$. 
If $u_{2i'+1}^{(1)} \in S$, then either $u_{2i'+1} \in\delta(S)$ or $u_{2i'+1} \in S$. 
If $u_{2i'+1} \in S$, then either $u_{2i'} \in\delta(S)$ or $u_{2i'} \in S$. 
If $u_{2i'} \in S$, then $u_{2i'}^{(2)} \in\delta(S)$, since $u_{2i'}^{(2)} \not\in S$ by condition (1) of $P(i+1)$. 
It then follows that $|\{u_{2i'+1}^{(1)}, u_{2i'+1}, u_{2i'}, u_{2i'}^{(2)}\}\cap \delta(S)|\geq 1$, and therefore that 
\[
    \Big|\delta(S) \cap \left( \bigcup_{0 \leq i' <D-1} \{u_{2i'+1}^{(1)}, u_{2i'+1}, u_{2i'}, u_{2i'}^{(2)}\} \right)\Big| \geq D-1.
\]
It follows that $|\delta(S)\setminus \Desc(u_{2D-2})| \geq D-1$. 

From condition (3) of $P(D-1)$, we know $\Desc(u_{2D-2})\cap S$ contains at least $(t_{2D} -1) -3D+3 \geq 2^{\log{3d}}-1 -3d+3 = 3d-3D+2 \geq 2$ vertices and contains at most $(t_{2D}-1)$ vertices, and so $|\delta(S) \cap \Desc(u_{2D-2})| \geq 1$. 
We then have that $|\delta(S)| \geq D$. 
We note that 
\begin{eqnarray*}
D&\geq &\frac{d - \lceil\log_2(3d)\rceil}{2} -1 \\
&\geq & \frac{d - \log_2(3d)-1}{2} -1 \\
&= & \frac{d - \log_2{d}}{2} -\frac{3+\log_2{3}}{2},
\end{eqnarray*}
and the proof follows.
\end{proof}

\subsection{Upper bounds}
\label{sec:upperbounds}

A \emph{linear layout} is a bijection $L:V(G)\rightarrow \{1,2, \ldots, |V(G)|\}$, and define $\vs_L(G) = \max_i |\delta(\{u\in V(G) : L(u)\leq i\})|$. The \emph{vertex separation number} of $G$, introduced in \cite{MR0756435}, is $$\vs(G) = \min \{ \vs_L(G) : L \text{ is a linear layout of } G\}.$$ In a linear layout $L$ of graph $G$, $|\delta(\{u\in V(G) : L(u)\leq i\})| \geq \Phi_V(G,i)$, and so it follows that 
$\Phi_V(G) \leq \vs(G).$ 
This result is folklore, which has been assumed in previous works, but we state it explicitly here. 

\begin{theorem}\label{thm:phiLEQvs}
For a graph $G$, $\Phi_V(G) \leq \vs(G).$
\end{theorem}

Pathwidth is a well-known graph parameter, denoted by $\pw(G)$, which is equal to $\vs(G);$ see \cite{MR1178214}. Several results on the pathwidth of complete $q$-ary trees are known, although the literature sometimes phrases these in terms of other graph parameters. 
The cases $q=2$ and $q=3$ were originally found in \cite{MR0881212,MR1004243}. 
In Theorem~3.1 of \cite{MR1283019}, a general process was given to find the pathwidth of a tree, from which the remaining cases can straightforwardly be derived. 
This can be combined with Theorem~\ref{thm:phiLEQvs} to yield the following result. 
\begin{theorem}\cite{MR1283019, MR0881212, MR1004243} \label{thm:pathwidth}
Let $T$ be a $q$-ary tree of depth $d$. If $q\geq 3$, then we have that $\Phi_V(T) \leq \mathrm{pw}(T) = d,$ and if $q=2$, 
   $\Phi_V(T) \leq \mathrm{pw}(T) = \lceil d/2 \rceil.$
\end{theorem}

Both the proofs of Theorem~\ref{thm:both_q5} and \ref{thm:both_q2} are now complete as a consequence of Theorems~\ref{thm:lower_q5}, \ref{thm:lower_q2}, and \ref{thm:pathwidth}; however, the upper bound of Theorem~\ref{thm:both_q34} is not yet established since the stated upper bound is stronger than that given in Theorem~\ref{thm:pathwidth}.
In the case that $q \in \{3,4\}$, we can strengthen the upper bound to differ from the lower bound of Theorem~\ref{thm:lower_q34} by an additive constant. The following result completes that proof of Theorem~\ref{thm:both_q34}. 

\begin{theorem}
    Suppose $T$ is a complete $q$-ary tree of depth $d$ with $q \in \{3,4\}$ and with $d$ sufficiently large. We then have that 
    \[\Phi_V(T) \leq d-\big\lfloor \log_q(d)\big\rfloor + 2 \leq d - \log_q(d)-3.\]
\end{theorem}
\begin{proof}
  For convenience, let $\alpha = \lfloor \log_q(d)\rfloor-2$. 
  Observe that with this value of $\alpha$ we have $d-2\alpha-1 \geq t_{d-\alpha-1}$ since 
  $t_{d-\alpha-1} 
  = \frac{q^{\alpha+2}-1}{q-1} 
  \leq \frac{d}{q-1}
  \leq (d-1)/2,$ where this last implication follows since $q \geq 3$, and since $(d-1)/2 \leq d - 2(\log_q(d) -2) -1 \leq d-2\alpha -1$, which is straightforwardly verified using calculus.

  We suppose for the sake of contradiction that $\Phi_V(T) > d - \alpha$, which gives that there is some integer $s$, with $1 \leq s \leq |V(T)|$, such that all subsets of vertices $S \subseteq V(T)$ with $|S|=s$ have $|\delta(S)| \geq d-\alpha +1$. 
  We will construct some set $S$ of cardinality $|S|=s$ over a series of steps and then modify this set to form $\overline{S}$ of cardinality $|\overline{S}|=s$. 
  To form the contradiction, we show that either $|\delta(S)| \leq d-\alpha$ or $|\delta(\overline{S})| \leq d-\alpha$. 
  We initialize $S=\emptyset$. 
    
  During step $i$ of the recursion, the value $r_i$ denotes the number of remaining vertices that must still be placed into $S$ to ensure the cardinality of $S$ will be $s$ at the end of the process. 
  As such, define $r_0=s$ since we will need to place another $s$ vertices into $ S=\emptyset$ for $S$ to have cardinality $s$. 
  
  Over the steps in this proof, we will define a path $(u_0, u_1, \ldots)$ with $u_0$ as the root vertex of the tree, and where $u_{i}$ is a child of $u_{i-1}$.
  It will be convenient to define four sets, $S_0$, $S_1'$, $S_1$, and $S_2$, which are each initialized as the empty set. 

  Suppose we are currently at step $i$ of the recursion. 
  We observe that we cannot have $r_{i-1}=0$, or else we would have stopped the process during the last round. 
  We split into four cases. 

\medskip
\noindent   {\emph Case 1}:  $1 \leq r_{i-1} \leq t_{i}-2$.  We define $s_i=0$, and we do not add vertices to the set $S$ during this step. 
    Add $u_{i-1}$ to $S_0$. 
    We set $r_i = r_{i-1}$ and set $u_i=u_{i-1}^{(1)}$. 

\medskip
\noindent   {\emph Case 2}: $t_{i}-1 \leq r_{i-1} \leq 2t_{i}-2$. 
    We define $s_i=1$. 
    Add the $t_{i}-1$ vertices of $\Desc(u_{i-1}^{(1)})\setminus \{u_{i-1}^{(1)}\}$ to $S$. 
    Add the vertex $u_{i-1}^{(1)}$ to $S_1'$ and the vertex $u_{i-1}$ to $S_1$. 
    If $r_{i-1} = t_{i}-1$, the process terminates at this step. 
    Otherwise, we set $r_i = r_{i-1}- (t_{i}-1)$ and set $u_i=u_{i-1}^{(2)}$. 

\medskip
\noindent   {\emph Case 3}: $r_{i-1} = 2t_{i}-1$. 
    Define $s_i=1$. 
    Add the $2t_{i}-1$ vertices of $\Desc(u_{i-1}^{(1)})  \cup \Desc(u_{i-1}^{(2)})\setminus \{u_{i-1}^{(2)}\}$ to $S$. 
    Add the vertices $u_{i-1}$ and $u_{i-1}^{(2)}$ to $S_2$. 
    We set $u_{i}=u_{i-1}^{(2)}$, continue to the next step, but then immediately terminate the procedure at this next step. 

\medskip
\noindent   {\emph Case 4}: $r_{i-1} \geq 2t_{i}$. 
    Define $s_i=\lfloor r_{i-1}/t_{i}\rfloor$.  
    Add the $s_it_{i}$ vertices $\bigcup_{1 \leq j \leq s_i} \Desc(u_{i-1}^{(j)})$ to $S$. 
    Add $u_{i-1}$ to $S_2$. 
    Note that there are $r_i=r_{i-1}-s_it_{i}$ vertices remaining that have not yet been assigned by the end of round $i$. 
    If $r_i=0$, the process terminates. 
    Otherwise, define $u_i=u_{i-1}^{(s_i+1)}$. 

  We make a number of observations about the result of this process: 
  \begin{enumerate}
    \item[(1)] The process terminates at or before round $d$;  
    \item[(2)] The path does not contain a leaf vertex and so has length at most $d$; 
    \item[(3)] $S$ is now of cardinality $s$; 
    \item[(4)] The set of vertices on the path $(u_0, u_1, \ldots)$ is the disjoint union $S_0 \cup S_1 \cup S_2$; 
    \item[(5)] $\delta(S)$ is the disjoint union $S_1' \cup S_2$; and
    \item[(6)] $|S_1'| = |S_1|$. 
  \end{enumerate}
  
  From (5) and (6), $|\delta(S)| = |S_1'| + |S_2| = |S_1| + |S_2|$. 
  Our initial assumption then yields that $d-\alpha+1 \leq |\delta(S)| = |S_1'| + |S_2| = |S_1| + |S_2|$
  From this equation, along with (2) and (4), it follows that $d \geq |S_0|+|S_1| + |S_2| \geq |S_0| + d-\alpha+1$, from which it follows that $\alpha-1 \geq |S_0|$. 

  We will now remove a number of vertices from $S$ to form the set $\overline{S}$. 
  Let $u\in V(T)$ be the right-most vertex of level $d-(\alpha+1)$ 
  such that $\Desc(u)\cap S\neq \emptyset$. 
  If the procedure to create $S$ lasted at least $d-(\alpha+1)$ 
  rounds, then $u = u_{d-(\alpha+1)}$. 
  Remove each vertex in $\Desc(u_{d-(\alpha+1)})\}$  
  from $S$, $S_0$, $S_1$, $S_1'$, and $S_2$ to yield the sets $\overline{S}$, $\overline{S}_0$, $\overline{S}_1$, $\overline{S}_1'$, and $\overline{S}_2$, respectively. 
  There are $t_{d-(\alpha+1) 
  } = \frac{q^{(\alpha+1) 
  +1}-1}{q-1}$ vertices in $\Desc(u_{d-(\alpha+1)
  })$, and so we have removed at most this many vertices from $S$. 
  
  We have that $|\overline{S}_1| + |\overline{S}_2| \geq |S_1| + |S_2| - (\alpha+2) 
  $, since vertices $\{u_i,u_i^{(1)}\}$ have been removed from $S_1$ and $S_2$ for only $\alpha+2$ 
  values of $i$. 
  Observing that $|\overline{S}_1| = |\overline{S}_1'|$, we have $|\overline{S}_1'| + |\overline{S}_2| \geq |S_1| + |S_2| - (\alpha+2)
  \geq (d-\alpha+1) - (\alpha+2) 
  = d-2\alpha - 1$. 
  Since $d-2\alpha-1 \geq t_{d-\alpha-1
  }$, it follows that there are more vertices in $\overline{S}_1'\cup \overline{S}_2$ compared to how many vertices were removed from $S$ to form $\overline{S}$.  
  For each vertex that we removed from $S$ to form $\overline{S}$, we add a vertex to $\overline{S}$ from $\overline{S}_1' \cup \overline{S}_2$, first adding vertices from $\overline{S}_1'$ to $\overline{S}$, and then once $\overline{S}_1'\subseteq \overline{S}$, adding vertices from $\overline{S}_2$ to $\overline{S}$.  

  For each vertex $u_i^{(1)} \in \overline{S}_1'$ that was added to $\overline{S}$, the vertex $u_i^{(1)} \in \delta(S)$ is no longer in $\delta(\overline{S})$, but its parent vertex, $u_i$, is now in $\delta(\overline{S})$. 
  No other vertex is added or removed from $\delta(\overline{S})$ by adding the vertex $u_i^{(1)}$ to $\overline{S}$.  
  For each vertex $u_i \in \overline{S}_2$ that was added to $\overline{S}$, the vertex $u_i\in \delta(S)$ is no longer in $\delta(\overline{S})$, but the vertices $u_{i-1}$, $u_{i+1}$, and $u_i^{(4)}$ might now be in $\delta(\overline{S})$. 
  No other vertex is added or removed from $\delta(\overline{S})$ by adding the vertex $u_i$ to $\overline{S}$. 

  If any $u_i$ is undefined for some $1 \leq i \leq d-(\alpha+1) 
  $, define it as the right-most child of $u_{i-1}$ such that $\Desc(u_i)\cap S \neq \emptyset$, and as $u_{i-1}^{(1)}$ if no such vertex exists. 
  Notice that $u_{d-(\alpha+1) 
  }$ might be in $\delta(\overline{S})$. 
  
  There are two cases to consider. 
    The first case is if there is some vertex $u_i^{(1)} \in \overline{S}_1'$ that was not added to $\overline{S}$, which implies no vertex $u_i \in \overline{S}_2$ was added to $\overline{S}$. 
      In this case, we may then observe that the border of $\delta(\overline{S})$ consists of the path vertices $\{u_0, \ldots, u_{d-(\alpha+1) 
      }\}$ and vertices with the form $u_i^{(1)}$, under the condition that $u_i^{(1)} \in \delta(\overline{S})$ implies that $u_i \notin \delta(\overline{S})$. 
    The second case assumes all vertices of $\overline{S}_1'$ were added to $\overline{S}$. 
      In this case, we may then observe that the border of $\delta(\overline{S})$ consists of the path vertices $\{u_0, \ldots, u_{d-(\alpha+1) 
      }\}$ and vertices with the form $u_i^{(4)}$, under the condition that $u_i^{(4)} \in \delta(\overline{S})$ implies that $u_i \notin \delta(\overline{S})$. 
  In both cases, this implies that $\delta(\overline{S})$ contains at most $d- \alpha
  $ vertices, contradicting our original assumptions. 
  Thus, we have found the contradiction, and the proof is complete.
\end{proof}

\section{Applications and Future Work}
\label{sec:applications}
For Cartesian grid graphs, the vertex separation number \cite{MR2397203} and the isoperimetric peak \cite{MR1082842} are equal. 
It is natural to ask whether these two properties are identical on all graphs. 
As far as we are aware, no example of a graph class has been found where $vs(G)\neq \Phi_V(G)$ until the work in this paper. We have shown that the two parameters may be arbitrarily far apart. The proof of the theorem follows directly from Theorem~\ref{thm:both_q2}.

\begin{theorem} \label{thm:vs}
    There exists a graph $G$ such that \[\vs(G) - \Phi_V(G) = \Omega\big(\log( \log(|V(G)|))\big).\] 
\end{theorem}

A \emph{nested solution} to the isoperimetric problem is a set of subsets of a graph, say $S_i$, with $S_{i} \subseteq S_{i+1}$ such that $|S_{i+1} \setminus S_i|=1$, and where $\Phi_V(G,i) = |\delta(S_i)|$. 
If a nested solution exists, then it follows that $vs(G)=\Phi_V(G)$. 
Using previous results that studied the isoperimetric problem for lattices \cite{MR0200192,MR2946103, WangWang}, 
Barber and Erde \cite{MR3819052} asked whether a nested solution always exists for Cayley graphs of $\mathbb{Z}^d$. 
This was answered in the negative in \cite{arXiv:2402.14087} although they also showed that once a certain threshold was reached, a phase transition occurred, and then the subsets did nest together. In particular, $S_{i} \subseteq S_{i+1}$ for $i$ larger than some particular value.  Although we do not provide a full isoperimetric profile for regular trees, it is not difficult to show from our work that the nesting property breaks frequently throughout. For example, in the case that $q=3$ and $d \geq 6$, there will be many occurrences where $\Desc(u_{d-5})$ contains $178$ vertices in $S$, implying that $\delta(S)\cap \Desc(u_{d-5}) = \Desc(u_{d-5}^{(1)})\cup\Desc(u_{d-4}^{(1)})\cup\Desc(u_{d-3}^{(1)})\cup\Desc(u_{d-2}^{(1)})$. 
When an additional vertex is added so that $\Desc(u_{d-5})$ contains $179$ vertices in $S$, we find that $\delta(S)\cap \Desc(u_{d-5}) = \Desc(u_{d-5}^{(1)})\cup\Desc(u_{d-4}^{(1)})\cup\Desc(u_{d-3}^{(1)})\cup\Desc(u_{d-2}^{(1)})\cup\Desc(u_{d-2}^{(2)})\setminus \{u_{d-5}^{(1)},u_{d-4}^{(1)},u_{d-3}^{(1)},u_{d-2}^{(1)}\}$. 
As such, no non-trivial phase transition occurs in the case of complete trees.

Attention has been paid to the `shape' of subsets of a graph that are isoperimetrically optimal \cite{MR3819052,MR4648905}, although most attention has been paid to the cases where the underlying graphs are lattices. 
Such sets are found to be close to some standard shape, which is a convex subset in the studied cases so far. 
We note the stark difference in the isoperimetrically optimal sets for complete trees we have found, as we have shown that isoperimetrically optimal subsets are often disconnected. 
While the shape in the previously studied cases was close to the shape formed by finding the first number of vertices traversed during a breadth-first transversal of the graph from a certain vertex, the shape formed for complete trees is the first number of vertices traversed during a depth-first traversal of the graph from the root vertex, where a vertex is included once the vertex is exited during the transversal.

There are several parameters bounded and related to pathwidth of a graph, such as the vertex bandwidth \cite{MR2410445}, network reliability \cite{MR1759753}, and the thinness of a graph \cite{MR3958231, MR4540048,MR2294675}. It would be interesting to further investigate how the isoperimetric parameter could influence these results. 

There have been a number of works that have used the bound of \cite{MR2579861}. In \cite{bonato2023kvisibility}, the authors used this result to provide a lower bound on the $k$-proximity number $\text{prox}_k(T)$, which itself is a lower bound on the $k$-visibility localization number $\zeta_k(T)$. Using a straightforward substitution, we find the following sharpened result. 

\begin{theorem}
For each $h \geq 3$, $q \geq 4$ and non-negative integer $k$, there exists a rooted tree of depth $d$ and maximum degree $q$ such that 
    \[
    \zeta_k(T) 
    \geq \mathrm{prox}_k(T) 
    > \frac{1}{4} \frac{d-(k+1)}{(2k+1)^2}.
    \]
\end{theorem}
This reduces the multiplicative difference between the known lower and upper bound for $\text{prox}_k$ on complete trees from $\frac{160}{3}$ down to $8$.  We note that the result from \cite{MR2579861} was also used in a result of \cite{MR4643801} for the one-visibility localization number of trees, and the above theorem supersedes this as well. 

We finish by mentioning that the total vertex separation number is a parameter similar to the vertex separation number, but the border sizes are summed rather than the maximum taken. This parameter has a number of interesting characteristics, such as being equal to the profile of the graph \cite{MR1669059} and being equal to the number of edges of the interval graph that contains $G$ as a subgraph and is the minimal such interval graph in terms of the number of edges \cite{MR1629604}.  The isoperimetric parameter in \cite{bonato2023kvisibility} is utilized as a special partial sum of the isoperimetric parameter. 
Given these two facts, it might be worthwhile investigating the full sum of the isoperimetric parameter: \[
\sum_{s=0}^{|V(G)|} \Phi(G,s),
\] 
or by replacing the terms of the sum with $\frac{\Phi(G,s)}{s}.$

\end{document}